\documentclass[12pt]{amsart}
\usepackage{amsmath,amsfonts,euscript,dsfont,amscd,amsthm,amssymb,upref,graphics}
\usepackage[all]{xy}

\theoremstyle{definition}

\swapnumbers

\theoremstyle{plain}

\newtheorem{theorem}{Theorem}[section]
\newtheorem{proposition}[theorem]{Proposition}
\newtheorem{lemma}[theorem]{Lemma}
\newtheorem{corollary}[theorem]{Corollary}
\newtheorem{observation}[theorem]{Observation}

\theoremstyle{definition}

\newtheorem{definition}[theorem]{Definition}

\newtheorem{parag}[theorem]{}
\newtheorem{example}[theorem]{Example}
\newtheorem{examples}[theorem]{Examples}
	
\newtheorem{notation}[theorem]{Notation}

\newtheorem{remark}[theorem]{Remark}

\theoremstyle{remark}

{\begin{enumerate}\setlength{\itemsep}{#1}}{\end{enumerate}}
\newenvironment{enumerata}%
{\begin{enumerate}
\renewcommand{\theenumi}{\alph{enumi}}
}{\end{enumerate}}
{\begin{enumerata}\setlength{\itemsep}{#1}}{\end{enumerata}}

\newcommand{\Spec}{	\operatorname{{\rm Spec}}}

\newcommand{\Proj}{	\operatorname{{\rm Proj}}}

\newcommand{\haut}{	\operatorname{{\rm ht}}}
\newcommand{\supp}{	\operatorname{{\rm supp}}}
\newcommand{\image}{	\operatorname{{\rm im}}}

\newcommand{\bideg}{	\operatorname{{\rm bideg}}}
\newcommand{\trdeg}{	\operatorname{{\rm trdeg}}}
\newcommand{\Frac}{	\operatorname{{\rm Frac}}}
\newcommand{\Char}{	\operatorname{{\rm char}}}

\newcommand{\Div}{	\operatorname{{\rm Div}}}

\renewcommand{\div}{	\operatorname{{\rm div}}}

\newcommand{\Miss}{	\operatorname{{\rm Miss}}}

\newcommand{\dic}{	\operatorname{{\rm dic}}}

\newcommand{\Rec}{	\operatorname{{\rm Rec}}}
\newcommand{\lco}{	\operatorname{{\rm lco}}}

\newcommand{\Galg}{	\operatorname{\Gamma_{\!\text{\rm alg}}}}

\newcommand{\setspec}[2]{\big\{\,#1\, \mid \,#2\, \big\}}

\newcommand{\Integ}{\ensuremath{\mathbb{Z}}}
\newcommand{\Nat}{\ensuremath{\mathbb{N}}}

\newcommand{\Comp}{\ensuremath{\mathbb{C}}}
\newcommand{\Reals}{\ensuremath{\mathbb{R}}}
\newcommand{\aff}{\ensuremath{\mathbb{A}}}
\newcommand{\proj}{\ensuremath{\mathbb{P}}}
\newcommand{\bk}{{\ensuremath{\rm \bf k}}}
\newcommand{\ck}{{\bar{\bk}}}
\newcommand{\kk}[1]{\bk^{[#1]}}

\newcommand{\bbV}{\ensuremath{\mathbb{V}}}

\newcommand{\pgoth}{{\ensuremath{\mathfrak{p}}}}

\newcommand{\qgoth}{{\ensuremath{\mathfrak{q}}}}
\newcommand{\mgoth}{{\ensuremath{\mathfrak{m}}}}

\newcommand{\Aeul}{\EuScript{A}}
\newcommand{\Beul}{\EuScript{B}}

\newcommand{\Meul}{\EuScript{M}}
\newcommand{\Neul}{\EuScript{N}}
\newcommand{\Oeul}{\EuScript{O}}

\newcommand{\Seul}{\EuScript{S}}

\newcommand{\isom}{\cong}
\renewcommand{\epsilon}{\varepsilon}
\renewcommand{\phi}{\varphi}
\renewcommand{\emptyset}{\varnothing}

\newlength{\mylength}
\newcommand{\abh}{%
\raisebox{-1.2mm}{%
\setlength{\unitlength}{1mm}%
\begin{picture}(4.4,4.4)(-2.2,-2.2)
\put(0,0){\circle{2}}
\put(-2,0){\line(1,0){4}}
\end{picture}}}

\newcommand{\rien}[1]{}

\CompileMatrices

\begin{document}
\renewcommand{\baselinestretch}{1.07}


\title[Very good and very bad field generators]{Very good and very bad field generators}

\author{Pierrette Cassou-Nogu\`es}
\author{Daniel Daigle}

\address{IMB, Universit\'e Bordeaux 1 \\
351 Cours de la lib\'eration, 33405, Talence Cedex, France}
\email{Pierrette.Cassou-nogues@math.u-bordeaux1.fr}

\address{Department of Mathematics and Statistics\\
University of Ottawa\\
Ottawa, Canada\ \ K1N 6N5}
\email{ddaigle@uottawa.ca}

\thanks{Research of the first author partially supported by Spanish grants MTM2010-21740-C02-01 and  MTM2010-21740-C02-02.}
\thanks{Research of the second author supported by grant RGPIN/104976-2010 from NSERC Canada.}

{\renewcommand{\thefootnote}{}
}

{\renewcommand{\thefootnote}{}
}

\begin{abstract} 
Let $\bk$ be a field.
A {\it field generator\/} is a polynomial $F \in \bk[X,Y]$ satisfying $\bk(F,G) = \bk(X,Y)$ for some
$G \in \bk(X,Y)$. If $G$ can be chosen in $\bk[X,Y]$, we call $F$ a {\it good field generator\/};
otherwise, $F$ is a {\it bad field generator}.
These notions were first studied by Abhyankar, Jan and Russell in the 1970s.
The present paper introduces and studies the notions of ``very good'' and ``very bad'' field generators.
We give theoretical results as well as new examples of bad and very bad field generators.
\end{abstract}

\maketitle
  
\vfuzz=2pt


\section{Introduction}
\label {SEC:Introduction}

Throughout this paper, $\bk$ denotes an arbitrary field unless otherwise specified.

\medskip
If $R$ is a subring of a ring $S$, the notation $S=R^{[n]}$ means that $S$ is $R$-isomorphic
to a polynomial algebra in $n$ variables over $R$.
If $L/K$ is a field extension, $L = K^{(n)}$ means that $L$ is a purely transcendental extension of $K$,
of transcendence degree $n$.
We write $\Frac R$ for the field of fractions of a domain $R$.
All curves and surfaces are irreducible and reduced.

\begin{definition} \label {0v23g7h23kjli320e}
Let $A = \kk2$ and $K = \Frac A$.
A {\it field generator of $A$} is an $F \in A$ satisfying $K = \bk(F,G)$ for some $G \in K$.
A {\it good field generator of $A$} is an $F \in A$ satisfying $K = \bk(F,G)$ for some $G \in A$.
A field generator which is not good is said to be {\it bad}.
\end{definition}

Field generators are studied in \cite{JanThesis}, \cite{Rus:fg}, \cite{Rus:fg2}
and \cite{Cassou-BadFG}.
The first example of bad field generator was given in \cite{JanThesis},
and more examples were given in \cite{Rus:fg2} and \cite{Cassou-BadFG}.
Among other things, \cite{Rus:fg2} showed that $21$ and $25$ are the smallest integers $d$
such that there exists a bad field generator of degree $d$.

The notions of good and bad field generators are classical. 
We shall now introduce the notions of  ``very good'' and ``very bad'' field generators.
Before doing so, let us first adopt a convention that we shall keep throughout this paper.
Namely,  let us agree that
the notation ``$A \preceq B$'' means that all the following conditions are satisfied:
$$
A = \kk2, \quad B = \kk2, \quad A \subseteq B \quad \text{and} \quad \Frac A = \Frac B .
$$

Observe that if $F \in A \preceq B$, then
$F$ is a field generator of $A$ iff $F$ is a field generator of $B$.
Moreover, if $F$ is a good field generator of $A$ then it is a good field generator of $B$
(and consequently, if it is a bad field generator of $B$ then it is a bad field generator of $A$).
However, it might happen that $F$ be a bad field generator of $A$ and  a good field generator of $B$.
These remarks suggest the following:

\pagebreak[3]
\begin{definition}  \label {Eir128372et12uhwjd}
Let $F \in A = \kk2$ be a field generator of $A$.
\begin{enumerate}

\item $F$ is a {\it very good\/} field generator of $A$ if it is a good field generator
of each $A'$ satisfying $F \in A' \preceq A$.

\item $F$ is a {\it very bad\/} field generator of $A$ if it is a bad field generator
of each $A'$ satisfying $A' \succeq A$.

\end{enumerate}
\end{definition}

It is interesting to note that the notion of very good field generator suggested itself in
a natural way, in our study \cite{CN-Daigle:Lean} of ``lean factorizations'' of morphisms $\aff^2 \to \aff^1$.
The definition of very bad field generator then follows by symmetry.

It is clear that ``very good'' implies ``good'' and that ``very bad'' implies ``bad''.
Examples of very good field generators are easy to find; moreover,
it follows from \ref{q23328r83yd74r9128}--\ref{pf9823f898012dj}
that certain well-studied classes of polynomials are included in that of very good field generators.
Paragraph \ref{GoodBadUgly} gives examples of very bad field generators, 
of bad field generators which are not very bad, and of good ones which are not very good.

\begin{notation}
Let $A = \kk2$.  Given $F \in A \setminus \bk$, we let
$\Galg(F,A)$ denote the set of prime ideals $\pgoth$ of $A$ such that
the composite $\bk[F] \hookrightarrow A \to A/\pgoth$ is an isomorphism.
We also let 
$\Gamma(F,A) = \setspec{ V(\pgoth) }{ \pgoth \in \Galg(F,A) }$,
i.e., $\Gamma(F,A)$ is the set of curves $C \subset \Spec A$ which have the property
that the composite $C \hookrightarrow \Spec A \to \Spec\bk[F]$ is an isomorphism.
Note that $\pgoth \mapsto V(\pgoth)$ is a bijection  $\Galg(F,A) \to \Gamma(F,A)$.
\end{notation}

The set $\Gamma(F,A)$ (or equivalently $\Galg(F,A)$) plays an important role
in our study of field generators.
One of our main results is Theorem~\ref{9854dnc2mrhvfc}, which asserts that if 
$F$ is a field generator of $A=\kk2$ satisfying $| \Gamma(F,A) | > 2$
then there exist $X,Y$ such that $A=\bk[X,Y]$ and $F = \alpha(Y)X+\beta(Y)$ for some
$\alpha(Y), \beta(Y) \in \bk[Y]$
(note that  $\alpha(Y)X+\beta(Y)$ is a very good field generator of an especially simple type).
In particular, if $F$ is a bad field generator of $A$ then $|\Gamma(F,A)| \in \{0,1,2\}$,
where (by \ref{92898fhuqnvalh893}) the three cases occur,
and where (by \ref{88x889adb823mdnfv}) the case  $|\Gamma(F,A)| = 0$ is equivalent 
to $F$ being very bad.
This last equivalence is a characterization of very bad field generators which turns out to be 
easy to use in practice.
A characterization of very good field generators is not known,
but \ref{q23328r83yd74r9128} is a partial result in that direction.

Section~\ref{SEC:VGVBFGs} also shows
how to construct very bad field generators from a given bad field generator.
We use that construction method in proofs (for instance in \ref{0d023epoi012p94938yte6q}) and also for
giving new examples of bad and very bad field generators (\ref{GoodBadUgly}).
Our aim, with these examples, is not to give explicit polynomials (which would be in principle easy),
but rather to demonstrate the method.

The main results are in sections \ref{SEC:CardGammaFG} and \ref{SEC:VGVBFGs},
but results \ref{89329d675fd43q} and \ref{difja;skdfj;aksd} are also noteworthy.

\medskip
We reiterate that
$\bk$ denotes an arbitrary field (unless otherwise specified) throughout this paper.
We write $\aff^n$ or $\aff_\bk^n$ for the affine $n$-space over $\bk$, i.e., the scheme $\Spec A$ where $A=\kk n$.

\section{Dicriticals}
\label {SEC:PrelimsPartII}

\begin{definition}  \label {difp23qwjksd}
Given a field extension $L \subseteq M$, let $\bbV(M/L)$ be the set of valuation rings $R$
satisfying $L \subseteq R \subseteq M$, $\Frac R = M$ and $R \neq M$.

Given a pair $(F,A)$ such that $A = \kk2$ and $F \in A \setminus \bk$, define
$$
\bbV^\infty(F,A)
= \setspec{ R \in \bbV( K /\bk(F) ) }{  A \nsubseteq R } \quad \text{where $K = \Frac A$.}
$$
Then $\bbV^\infty(F,A)$ is a nonempty finite set which depends only on the pair $(\bk(F), A)$.
For each $R \in \bbV^\infty(F,A)$, let $\mgoth_R$ be the maximal ideal of $R$.
Let $R_1, \dots, R_t$ be the distinct elements of $\bbV^\infty(F,A)$
and $d_i = [ R_i/\mgoth_{R_i} : \bk(F) ]$ for $i = 1, \dots, t$.
Then we define
$$
\Delta(F,A) = [ d_1, \dots, d_t ] \quad\text{and}\quad \dic(F,A) = | \bbV^\infty(F,A) | = t
$$
where $[ d_1, \dots, d_t ]$ is an unordered $t$-tuple of positive integers.

Given $A = \kk2$ and $F \in A \setminus \bk$,
we call the elements of  $\bbV^\infty(F,A)$ the {\it dicriticals\/} of $(F,A)$, or of $F$ in $A$;
given $R \in \bbV^\infty(F,A)$, we call $[ R/\mgoth_R : \bk(F) ]$ the {\it degree of the dicritical\/}~$R$.
\end{definition}

\begin{remark} \label {939ri09109e2xj9cru}
Let $A = \kk2$ and  $F \in A \setminus \bk$.
Choose a pair $\gamma = (X,Y)$ satisfying $A = \bk[X,Y]$ and consider the embedding of $\aff^2$ in $\proj^2$,
$(x,y) \mapsto (1:x:y)$, determined by $\gamma$.
That is, identify $\aff^2$ with the complement of the line ``$W=0$'' in 
$\proj^2 = \Proj\bk[W,X,Y]$, where $\bk[W,X,Y] = \kk3$ is $\Nat$-graded by total degree in $W,X,Y$.
Consider the closed subset $V(F)$ of $\aff^2$ and its closure $\overline{V(F)}$ in $\proj^2$.

For each $R \in \bbV^\infty(F,A)$, there exists a unique point $Q_R \in \proj^2$
such that $R$ is centered at $Q_R$ (i.e., $R$ dominates the local ring of $\proj^2$ at $Q_R$).
One can see that $R \mapsto Q_R$ is a surjective set map
$ \bbV^\infty(F,A) \to \overline{V(F)} \setminus V(F)$.
It follows that 
\begin{equation} \label {89cxcubvcaw1ewxdfey}
\dic(F,A)\ \ge\  | \overline{V(F)} \setminus V(F) | .
\end{equation}
Note that \eqref{89cxcubvcaw1ewxdfey} is valid for every choice of $\gamma=(X,Y)$.
The right hand side of \eqref{89cxcubvcaw1ewxdfey}  depends on $(F,A,\gamma)$,
but $\dic(F,A)$ depends only on $(F,A)$.
\end{remark}

\begin{remark}  \label {difqwkae.dmnkdfuuuu6}
Except for the notations,
our definitions of ``dicritical'' and of ``degree of dicritical'' are identical to those 
given by Abhyankar in \cite{Abh:DicDiv2010} (see the last sentence of page 92).
Note, however, that many authors use a definition formulated in terms of horizontal curves at infinity.
Let us make the link between those two approaches. 
For this discussion, we assume that $\bk$ is algebraically closed.
Consider a pair $(F,A)$ such that $A = \kk2$ and $F \in A \setminus \bk$.
Let us use the abbreviation $\bbV = \bbV( \Frac(A) /\bk(F) )$;
then $\bbV^\infty(F,A) = \setspec{ R \in \bbV }{  A \nsubseteq R }$.

Let $f : \aff^2 = \Spec A \to \aff^1 = \Spec\bk[F]$ be the morphism determined by the inclusion $\bk[F] \to A$.
Then there exists a (non unique) commutative diagram
\begin{equation}  \label{dkfjasodfla2}
\xymatrix{
\aff^2 \ar[d]_{f} \ar @<-.4ex> @{^{(}->}[r]  &  X \ar[d]^{\bar f}  \\
\aff^1 \ar @<-.4ex> @{^{(}->}[r] & \proj^1
}
\end{equation}
where $X$ is a nonsingular projective surface, the arrows ``$\hookrightarrow$'' are open immersions and
$\bar f$ is a morphism.
Let us say that a curve $C \subset X$ is ``horizontal'' if it satisfies $\bar f(C) = \proj^1$,
let $H$ denote the set of curves $C \subset X$ which are horizontal and let
$H^\infty = \setspec{ C \in H }{ C \subseteq X \setminus \aff^2 }$.
Several authors refer to the elements of $H^\infty$ as 
the dicriticals of $f$.
For each $C\in H^\infty$, the degree of the morphism $\bar f|_C : C \to \proj^1$
is then called the ``degree of the dicritical'' $C$.

Given $C \in H$, let $\xi_C$ be the generic point of $C$ and observe that the local ring $\Oeul_{X,\xi_C}$ of $X$ at
the point $\xi_C \in X$ is an element of $\bbV$. In fact the map
$C \mapsto \Oeul_{X,\xi_C}$ from $H$ to  $\bbV$ is bijective, and so is its restriction
$H^\infty \to \bbV^\infty(F,A)$.  Thus 
$$
\dic(F,A)=|H^\infty|,
$$
where $\dic(F,A)$ is defined in \ref{difp23qwjksd}.
We also note that if $C \in H$ corresponds to $R = \Oeul_{X, \xi_C} \in \bbV$ by the above bijection $H \to \bbV$,
then the degree of $\bar f|_C : C \to \proj^1$ is equal to $[ \bk(C) : \bk( \proj^1 ) ]$,
where the function field $\bk(C)$ of $C$ can be identified with the residue field of $R$; so
$\deg(\bar f|_C) = [ R/\mgoth_R : \bk(F) ]$.
Consequently, if we write $H^\infty = \{ C_1, \dots, C_t \}$ then
$$
\Delta(F,A) = \big[ \deg(\bar f|_{C_1} ),\dots, \deg(\bar f|_{C_t} ) \big] ,
$$
where $\Delta(F,A)$ is defined in \ref{difp23qwjksd}.
To summarize, our definition (\ref{difp23qwjksd}) of dicriticals and of their degrees is 
equivalent (via the bijection  $H^\infty \to \bbV^\infty(F,A)$, $C \mapsto \Oeul_{X,\xi_C}$) 
to that given in terms of $H^\infty$.

Let us be precise about the use of language (regarding dicriticals or some equivalent concept) in
\cite{BartoCassou:RemPolysTwoVars}, \cite{Cassou-BadFG}, \cite{Rus:fg} and \cite{Rus:fg2},
since we are going to refer to those papers.

Papers \cite{BartoCassou:RemPolysTwoVars} and \cite{Cassou-BadFG} follow
the $H^\infty$ approach, but do not use the word ``dicritical''.
The elements of $H^\infty$ are called ``horizontal curves'' or ``horizontal components'',
and the degree of such a curve $C$ is defined to be the degree of $\bar f|_C$.

Papers \cite{Rus:fg} and \cite{Rus:fg2} simply speak of points at infinity, instead of dicriticals.
Let us explain this.
Let $\bk$ be any field, $A=\kk2$, $F \in A \setminus \bk$, $L=\Frac A$ and $K=\bk(F)$.
Then $L/K$ is a function field in one variable.
Let $C_F$ be the complete regular curve over $K$ whose function field is $L$.
Let $P$ be a closed point of $C_F$ and $(\Oeul_P,\mgoth_P)$ the local ring of $C_F$ at $P$;
by definition of the degree of a point on a variety, $\deg P = [\Oeul_P/\mgoth_P : K]$;
if $A \nsubseteq \Oeul_P$, one says (in \cite{Rus:fg} and \cite{Rus:fg2}) that $P$ is a ``point of $C_F$ at infinity''
(or a ``place of $C_F$ at infinity'').
As is well known, the set of closed points of $C_F$ can be identified with $\bbV(L/K)$
via the bijection $P \mapsto \Oeul_P$;
then the set $\bbV^\infty(F,A)$ of dicriticals is precisely the set of points of $C_F$ at infinity
and the degrees of the dicriticals are the degrees of the points.  
For instance the sentence
``$C_F$ has exactly two places at infinity, one of degree $2$ and one of degree $3$'',
in \cite[p.\ 324]{Rus:fg2}, means that $\Delta(F,A)=[2,3]$.

This closes the remark on terminology.
Our terminology, throughout, is that of \ref{difp23qwjksd}.
\end{remark}

A {\it function field in one variable\/} is a finitely generated field extension of
transcendence degree $1$.

\begin{notation} \label {2z5vbtf7e4jhr}
Given a function field in one variable $L/K$, we set
$$
\eta(L/K) = \gcd \setspec{ [ R/\mgoth_R : K ] }{ R \in \bbV(L/K) }
$$
where $\bbV(L/K)$ is defined in \ref{difp23qwjksd}.
Note that if $L/K$ has a $K$-rational point\footnote{A $K$-rational point is an $R \in \bbV(L/K)$
satisfying $R/\mgoth_R=K$.} then $\eta(L/K)=1$.
In particular, if $L = K^{(1)}$ then $\eta(L/K)=1$.
\end{notation}

\begin{theorem}  \label {89329d675fd43q}
Let $\bk$ be a field, $A=\kk2$, $F \in A \setminus \bk$ and $\Delta(F,A) = [d_1, \dots, d_t]$.
Then
$$
\gcd(d_1, \dots, d_t) = \eta \big(\Frac(A)/\bk(F) \big).
$$
In particular, if $F$ is a field generator of $A$ then $\gcd(d_1, \dots, d_t) =1$.
\end{theorem}

\begin{proof}
Let $\Aeul = S^{-1}A$, where $S = \bk[F] \setminus \{0\}$;
let $K=\bk(F)$ and $L = \Frac A = \Frac \Aeul$.
As $\Aeul$ is a domain and a finitely generated $K$-algebra, its Krull dimension
is $\dim\Aeul = \trdeg(\Aeul/K)=1$; being a $1$-dimensional UFD, $\Aeul$ is a PID.

Let $R_1, \dots, R_t$ be the distinct elements of the subset $\bbV^\infty(F,A)$ of $\bbV(L/K)$,
let $d_i = [ R_i/\mgoth_{R_i} : K]$ for $1 \le i \le t$ and let $d = \gcd(d_1,\dots,d_t)$.
It is obvious that $\eta(L/K) \mid d$.
So it's enough to show that $[R/\mgoth_R : K] \in d\Integ$ for every $R\in\bbV(L/K)$.

If $R \in \bbV^\infty(F,A)$ then $[R/\mgoth_R : K] \in \{ d_1, \dots, d_t \}$, so 
$[R/\mgoth_R : K] \in d\Integ$ is obvious.

Let $R \in \bbV(L/K) \setminus \bbV^\infty(F,A)$.
Then $A \subseteq R$, so $\Aeul \subseteq R$ and consequently
$\mgoth_R \cap \Aeul$ is a prime ideal of $\Aeul$.
If $\mgoth_R \cap \Aeul=0$ then $\Frac \Aeul \subseteq R$, which contradicts the definition of $\bbV(L/K)$.
So $\mgoth_R \cap \Aeul$ is a maximal ideal $\mgoth$ of $\Aeul$; then $R = \Aeul_\mgoth$.
Since $\Aeul$ is a PID, $\mgoth=(g)$ for some $g \in \Aeul\setminus\{0\}$.
Multiplying $g$ by a suitable element of $K^*$, we may (and we do) arrange that $g \in A\setminus\{0\}$.

As $g \in L^*$, we may consider the principal divisor
$\div(g) \in \Div(L/K)$, where $\Div(L/K)$ is the free abelian group on the set $\bbV(L/K)$, written additively.
Note that (i)~$g$ is a uniformizing parameter of $R$;
(ii)~since $g\in A$, $g$ belongs to each element of $\bbV(L/K) \setminus \bbV^\infty(F,A)$;
and (iii)~if $R'$ is any element of $\bbV(L/K) \setminus \bbV^\infty(F,A)$ satisfying $g \in \mgoth_{R'}$,
then $\mgoth_{R'} \cap \Aeul=\mgoth$ and consequently $R'=\Aeul_\mgoth=R$.
Thus
$$
\div(g) = 1 R + \textstyle \sum_{i=1}^t v_i(g) R_i,
$$
where $v_i$ is the valuation of $R_i$. Since $\div(g)$ has degree $0$,
$$
0 = \big[R / \mgoth_R : K \big] +  \textstyle \sum_{i=1}^t v_i(g) \big[ R_i/\mgoth_{R_i} : K]
= \big[R/\mgoth_R : K \big] +  \textstyle \sum_{i=1}^t v_i(g) d_i,
$$
so $\big[R/ \mgoth_R : K \big] \in d\Integ$.
It follows that $d \mid \eta(L/K)$ and hence that $d=\eta(L/K)$, as desired.

If $F$ is a field generator of $A$ then $L = K^{(1)}$, so $\eta(L/K)=1$ and consequently $d=1$.
\end{proof}

Our next objective is to study how $\Delta(F,A)$ behaves under a birational extension of $A$.
Before doing this, we first need to discuss birational morphisms $\aff^2 \to \aff^2$.

\begin{definition}
Let $\Phi : X \to Y$ be a morphism of nonsingular algebraic surfaces over $\bk$.
Assume that $\Phi$ is {\it birational}, i.e., that there exist nonempty Zariski-open subsets $U \subseteq X$
and $V \subseteq Y$ such that $\Phi$ restricts to an isomorphism $U\to V$.
By a {\it missing curve\/} of $\Phi$ we mean a curve $C \subset Y$
such that $C \cap \Phi(X)$ is a finite set of closed points.
We write $\Miss(\Phi)$ for the set of missing curves of $\Phi$. Note that $\Miss(\Phi)$ is a finite set.
\end{definition}

\begin{lemma} \label {0923092cn3r98z29eu}
Let $\Phi : \aff^2 \to \aff^2$ be a birational morphism.
\begin{enumerata}

\item \label {90dgh3576tjv0498ugjh}
Let $C$ be a missing curve of $\Phi$.
Then $C$ has one place at infinity and, consequently,
the units of the coordinate algebra of $C$ are algebraic over $\bk$.

\item \label {0as9a0f9s9s3487jhh3}
Let $C \subset \aff^2$ be a curve which is not a missing curve of $\Phi$.
Then there exists a unique curve $D \subset \aff^2$ such that $\Phi(D)$ is a dense subset of $C$.

\end{enumerata}
\end{lemma}

\begin{proof}
First consider the case where $\bk$ is algebraically closed.
Then \eqref{90dgh3576tjv0498ugjh} is true by \cite[4.3]{Dai:bir}.
For \eqref{0as9a0f9s9s3487jhh3},
we consider \cite[1.1]{Dai:bir} (with $f=\Phi$, $X=\aff^2$ and $Y=\aff^2$)
and we observe that,
given a curve $C \subset Y$ which is not a missing curve of $f$,
$\tilde C \cap X$ (where $\tilde C \subset Y_n$ is the strict transform of $C$)
is the only curve in $X$ whose image in $Y$ is a dense subset of $C$.
This proves \eqref{0as9a0f9s9s3487jhh3}.
Our task, then, is to show that the Lemma continues to be valid when $\bk$ is an arbitrary field.

Let $A = \kk2$,  $\aff^2 = \Spec A$
and let $\phi : A \to A$ be the $\bk$-homomorphism corresponding to the given $\Phi : \Spec A \to \Spec A$.
Let $\ck$ be an algebraic closure of $\bk$.
Applying $(\ck \otimes_\bk\, \underline{\ \ })$ to $\phi$ gives commutative diagrams
$$
\xymatrix@R=10pt{
\bar A  &  \bar A \ar[l]_{\bar\phi}  \\
A \ar @{^{(}->} [u]   &  A \ar @{^{(}->} [u]   \ar[l]^{\phi}
} \qquad 
\xymatrix@R=10pt{
\Spec\bar A \ar[r]^{\bar\Phi} \ar[d]_{\alpha}  &  \Spec\bar A \ar[d]^{\alpha}  \\
\Spec A \ar[r]_{\Phi}  &  \Spec A \\
} \qquad \raisebox{-5mm}{$\alpha \circ \bar\Phi = \Phi \circ \alpha$}
$$

where $\bar A = \ck \otimes_\bk A = \ck^{[2]}$ and $\bar\Phi$ is a birational morphism.

(a) Since $\bar A$ is integral over $A$, we may choose a curve $\bar C \subset \Spec\bar A$ such that
$\alpha( \bar C ) = C$.
Then $\bar C$ is a missing curve of $\bar \Phi$, so, by \cite[4.3(c)]{Dai:bir}, $\bar C$ is a rational
curve with one place at infinity.
It follows that the coordinate algebra $S$ of $\bar C$ satisfies $S^* = \ck^*$.
As the coordinate algebra $R$ of $C$ satisfies $R \subseteq S$,
we deduce that $C$ has one place at infinity and that $R^* \subseteq S^* = \ck^*$,
which proves (a).

For the proof of (b),
let $G$ be the group of $A$-automorphisms of $\bar A$.
It follows from \cite[Th.\ 5, p.\ 33]{Matsumura} that, for each $\pgoth \in \Spec A$,
the subset $\alpha^{-1}(\pgoth)$ of $\Spec \bar A$ is equal to an orbit of the action of $G$
on $\Spec\bar A$.

Observe that, for any $\pgoth \in \Spec A$, the condition
``$V(\pgoth)$ is a curve in $\Spec A$ which is not a missing curve of $\Phi$''
is equivalent to ``$\haut\pgoth=1$ and $\pgoth \in \image(\Phi)$''
(we are using the fact that $\image(\Phi)$ is a constructible subset of $\Spec A$).
Moreover, if $\pgoth$ satisfies these equivalent conditions and $\qgoth \in \Phi^{-1}(\pgoth)$,
then $\haut\qgoth=1$ and $\Phi$ maps the curve $V(\qgoth)$ to a dense subset of $V(\pgoth)$.
Thus assertion (b) can be stated as follows:
{\it for each $\pgoth \in \Spec A$ such that $\haut\pgoth=1$, $|\Phi^{-1}(\pgoth)| \le 1$}.
Let $\pgoth \in \Spec A$ be such that $\haut\pgoth=1$ and let $\qgoth_1,\qgoth_2 \in \Phi^{-1}(\pgoth)$.
Since $\bar A$ is integral over $A$, we may choose $\bar\qgoth_1, \bar\qgoth_2 \in \Spec \bar A$ such that 
$\alpha(\bar\qgoth_i) = \qgoth_i$ for $i=1,2$.
Then $\alpha( \bar\Phi( \bar\qgoth_i )) = \Phi( \alpha( \bar\qgoth_i )) = \Phi( \qgoth_i ) = \pgoth$ ($i=1,2$),
i.e., $\bar\Phi( \bar\qgoth_1 )$ and $\bar\Phi( \bar\qgoth_2 )$ lie over $\pgoth$.
Then, by \cite[Th.\ 5, p.\ 33]{Matsumura},
there exists $\Theta \in G$ such that $\Theta ( \bar\Phi( \bar\qgoth_1 )) =  \bar\Phi( \bar\qgoth_2 )$.
As $\Theta \circ \bar\Phi = \bar\Phi \circ \Theta$, we have
$\bar\Phi( \Theta ( \bar\qgoth_1 )) =  \bar\Phi( \bar\qgoth_2 )$.
If we define $\bar\pgoth =  \bar\Phi( \bar\qgoth_2 )$ then
$\{ \Theta ( \bar\qgoth_1 ) , \bar\qgoth_2 \} \subseteq \bar\Phi^{-1}( \bar\pgoth )$;
as $\alpha(\bar\pgoth) =\alpha( \bar\Phi( \bar\qgoth_2 )) = \Phi( \alpha( \bar\qgoth_2 )) = \Phi( \qgoth_2 )=\pgoth$
and $\haut\pgoth=1$, we have $\haut\bar\pgoth=1$; by the case ``$\bk=\ck$\,'' of (b),
$|\bar\Phi^{-1}( \bar\pgoth )| \le 1$, so $\Theta ( \bar\qgoth_1 ) =  \bar\qgoth_2$;
consequently, $\alpha( \bar\qgoth_1 ) =  \alpha(\bar\qgoth_2)$, i.e., $\qgoth_1=\qgoth_2$.
\end{proof}

\begin{notation}  \label {293832476hf}
Consider morphisms $\aff^2 \xrightarrow{ \Phi } \aff^2 \xrightarrow{ f } \aff^1$
where $\Phi$ is birational and $f$ is dominant.
Then we write
\begin{align*}
\Miss_{\text{\rm hor}}(\Phi,f)  & = \setspec{ C \in \Miss(\Phi) }{ \text{$f(C)$ is a dense subset of $\aff^1$} }.
\end{align*}
We refer to the elements of $\Miss_{\text{\rm hor}}(\Phi,f)$
as the ``$f$-horizontal'' missing curves of $\Phi$.
\end{notation}

Note that, in the following result, $\Miss_{\text{\rm hor}}(\Phi,f)$ may be empty.
See the introduction for the notation ``$A \preceq B$''.

\begin{proposition}  \label {difja;skdfj;aksd}
Let $A \preceq B$ and $F \in A \setminus \bk$,
and consider the morphisms
\begin{equation*} \label  {8231h4g5hfhhcncbbxm}
\Spec B \xrightarrow{\Phi} \Spec A \xrightarrow{f} \Spec \bk[F]
\end{equation*}
determined by the inclusions $\bk[F] \hookrightarrow A  \hookrightarrow B$.
Let $C_1, \dots, C_h$ be the distinct elements of  $\Miss_{\text{\rm hor}}(\Phi,f)$ and,
for each $i \in \{ 1, \dots, h \}$,
let $\delta_i$ be the degree\footnote%
{Let $R \subseteq S$ be integral domains and $f : \Spec S \to \Spec R$ the corresponding morphism of schemes.
Assume that $\Frac S$ is a finite extension of $\Frac R$.  Then we define $\deg f = [\Frac S : \Frac R]$.}
of the morphism $f|_{C_i} : C_i \to \Spec\bk[F]$.
\begin{enumerata}

\item $\Delta(F,B) = \big[ \Delta(F,A), \delta_1, \dots, \delta_h \big]$,
i.e., $\Delta(F,B)$ is the concatenation of $\Delta(F,A)$ and $[\delta_1, \dots, \delta_h ]$.
In particular, $\dic(F,B) = \dic(F,A) + |\Miss_{\text{\rm hor}}(\Phi,f)|$.

\item For each $i \in \{1, \dots, h\}$, $\delta_i=1 \Leftrightarrow C_i \in \Gamma(F,A)$.

\end{enumerata}
\end{proposition}

\begin{proof}
(a)
Given local domains $(\Oeul,\mgoth)$ and $(\Oeul',\mgoth')$, we write
$\Oeul\unlhd\Oeul'$ to indicate that $\Oeul$ is a subring of $\Oeul'$ and that $\mgoth' \cap \Oeul = \mgoth$
(i.e., $\Oeul'$ dominates $\Oeul$).
Note that if 
$\Oeul\unlhd\Oeul'$, $\Frac\Oeul = \Frac\Oeul'$ and $\Oeul$ is a valuation ring, then $\Oeul=\Oeul'$.
Let $K = \Frac A = \Frac B$ and, for each $i \in \{ 1, \dots, h \}$,
let $\pgoth_i \in \Spec A$ be the generic point of $C_i$.
We claim:
\begin{equation} \label {c89j21mnbxcgs236o8}
\bbV^\infty(F,B) =  \bbV^\infty(F,A) \cup \big\{ A_{\pgoth_1}, \dots, A_{\pgoth_h} \big\}.
\end{equation}
It is obvious that $\bbV^\infty(F,A) \subseteq \bbV^\infty(F,B)$.
Let $i \in \{ 1, \dots, h \}$.
As $C_i$ is $f$-horizontal, $\pgoth_i \cap \bk[F] = 0$ and hence $A_{\pgoth_i} \in \bbV(K/\bk(F))$.
Since $C_i \in \Miss(\Phi)$, no $\qgoth \in \Spec B$ satisfies $\qgoth \cap A = \pgoth_i$,
so $B \nsubseteq A_{\pgoth_i}$ and hence $A_{\pgoth_i} \in \bbV^\infty(F,B)$. 
This proves ``$\supseteq$'' in \eqref{c89j21mnbxcgs236o8}.

Consider $R \in \bbV^\infty(F,B) \setminus \bbV^\infty(F,A)$.
Then $A \subseteq R$, so $\mgoth_R \cap A$ is an element $\pgoth$ of $\Spec A$.
As $\bk(F) \subset R$, $\pgoth \cap \bk[F] = 0$; so $\haut\pgoth<2$.
If $\pgoth = 0$ then $\Frac A \subseteq R$, which contradicts $R \in \bbV(K/\bk(F))$;
so $\haut\pgoth=1$.
Then $A_\pgoth \unlhd R$ where $A_\pgoth$ is a valuation ring, so $A_\pgoth = R$.
It also follows that $C = V(\pgoth)$ is a curve in $\Spec A$.
If $C \notin \Miss(\Phi)$ then there exists 
$\qgoth \in \Spec B$ satisfying $\qgoth \cap A = \pgoth$; then
$R=A_\pgoth \unlhd B_\qgoth$, so $B_\qgoth=R$ and hence $B \subseteq R$,
a contradiction; so $C \in \Miss(\Phi)$.
As $\pgoth \cap \bk[F]=0$, $C \in \Miss_{\text{\rm hor}}(\Phi,f)$
and consequently $\pgoth=\pgoth_i$ for some $i \in \{ 1, \dots, h \}$.
Then $R = A_{\pgoth_i}$ and \eqref{c89j21mnbxcgs236o8} is proved.

What is meant by $f|_{C_i} : C_i \to \Spec\bk[F]$ is really $\Spec(A/\pgoth_i) \to \Spec\bk[F]$, i.e.,
the morphism determined by the injective homomorphism $\bk[F] \to A/\pgoth_i$.
Let us abuse notation and write $\bk[F] \subseteq A/\pgoth_i$.
Then
$$
\delta_i = \deg(f|_{C_i}) = \big[ \Frac (A/\pgoth_i) : \bk(F) \big]
= \big[ A_{\pgoth_i}/ \pgoth_i A_{\pgoth_i} : \bk(F) \big]
$$
for each $i \in \{ 1, \dots, h \}$.
It follows that $\Delta(F,B) = \big[ \Delta(F,A), \delta_1, \dots, \delta_h \big]$,
because (clearly) \eqref{c89j21mnbxcgs236o8} is a disjoint union
and $A_{\pgoth_1}, \dots, A_{\pgoth_h}$ are distinct.
So (a) is proved.

(b) Let $i \in \{1, \dots, h\}$ and consider $\bk[F] \subseteq A/\pgoth_i$.
The condition $C_i \in \Gamma(F,A)$ is equivalent to $\bk[F] = A/\pgoth_i$,
and $\delta_i=1$ is equivalent to $\bk(F) = \Frac( A/\pgoth_i )$.
So it is clear that $C_i \in \Gamma(F,A)$ implies $\delta_i=1$.
Conversely, assume that $\delta_i=1$.
Then $\bk[F] \subseteq A/\pgoth_i \subset \bk(F)$, which implies that $A/\pgoth_i$ is a localization of $\bk[F]$.
Since the units of $A/\pgoth_i$ are algebraic over $\bk$ by \ref{0923092cn3r98z29eu},
the localization must be trivial, so $\bk[F]=A/\pgoth_i$ and consequently $C_i \in \Gamma(F,A)$. So (b) is proved.
\end{proof}

\section{Rectangular polynomials; properties of the set $\Gamma(F,A)$}
\label {SEC:RecGamma}

\begin{definition}
Let $A = \kk2$ and $\aff^2 = \aff^2_\bk = \Spec A$.
\begin{enumerate}

\item[(i)] A {\it variable\/} of $A$ is an element $F \in A$ satisfying $A = \bk[F,G]$ for some $G \in A$.

\item[(ii)]
A curve $C \subset \aff^2$ is called a {\it coordinate line\/} if $C=V(F)$ for some variable $F$ of~$A$.

\end{enumerate}
\end{definition}

\begin{remark}  \label {f02939d9wid29}
A curve $C \subset \aff^2$ is called a {\it line\/} if $C \isom \aff^1$.
It is clear that every coordinate line is a line,
and the Abhyankar-Moh-Suzuki Theorem (\cite{A-M:line}, \cite{Suzuki})
states that the converse is true if $\Char\bk=0$.
It is known that not all lines are coordinate lines if $\Char\bk > 0$
(on this subject, see e.g.\ \cite{Ganong:Survey} for a survey).
\end{remark}

\begin{definition} \label {90ojhgdew33689jhhf}
Let $A = \kk2$.
\begin{enumerate}

\item Given $F \in A$ and a pair $\gamma = (X,Y)$ such that $A = \bk[X,Y]$,
write $F = \sum_{i,j} a_{ij} X^iY^j$ where $a_{ij} \in \bk$ for all $i,j$;
then $\supp_\gamma(F) = \setspec{(i,j) \in \Nat^2}{ a_{ij} \neq 0 }$
is called the {\it support of $F$ with respect to $\gamma$}.

\item Given a subset $S$ of $\Reals^2$, let $\langle S \rangle$ denote its convex hull.

\item \label {3rfhbh8c98b82yh45ft6}
Given $F \in A$, we write $\Rec(F,A)$ for the set of ordered pairs $\gamma = (X,Y)$ satisfying 
$A = \bk[X,Y]$ and
$$
\text{there exist $m,n \ge 1$ such that 
$(m,n) \in \supp_\gamma(F) \subseteq \langle (0,0),(m,0),(0,n),(m,n) \rangle$.}
$$
Let $\Rec^+(F,A)$ be the set of $\gamma=(X,Y) \in \Rec(F,A)$ satisfying
the additional condition ``$m\le n$''.
Clearly, 
$$
\Rec^+(F,A) \neq \emptyset \Leftrightarrow \Rec(F,A) \neq \emptyset.
$$

\item By a {\it rectangular element of $A$} we mean an $F\in A$ satisfying $\Rec(F,A) \neq \emptyset$.

\end{enumerate}
\end{definition}

Some examples: no variable of $A=\kk2$ is rectangular;
the polynomial $X^2+Y^2$ is a rectangular element of $\Comp[X,Y]$ but not of $\Reals[X,Y]$.
See \ref{982393298urcnj092}, below,
to understand why the notion of rectangular element is relevant for studying field generators.

\begin{lemma}  \label {0cv9v8n8dllrklslk52kf0}
Let $F$ be a rectangular element of $A = \kk2$.
\begin{enumerata}

\item If $(X,Y) \in \Rec(F,A)$ then 
\begin{multline*}
\Rec(F,A) = \setspec{ (aX+b, cY+d) }{ a,b,c,d \in \bk,\ ac\neq0 } \\
\cup \setspec{ (cY+d, aX+b) }{ a,b,c,d \in \bk,\ ac\neq0 }.
\end{multline*}

\item \label {9834fh87td4rth}
Up to order, the pair $(m,n)$ in \ref{90ojhgdew33689jhhf}\eqref{3rfhbh8c98b82yh45ft6} depends only on $(F,A)$,
i.e., is independent of the choice of $\gamma \in \Rec(F,A)$.

\end{enumerata}
\end{lemma}

\begin{proof}
Consider $\gamma = (X,Y)$, $\gamma' = (U,V) \in \Rec(F,A)$ and $m,n,m',n' \ge1$ such that
\begin{gather}
\label {7655u7fghbvkaks02}  (m,n) \in \supp_\gamma(F) \subseteq \langle (0,0),(m,0),(0,n),(m,n) \rangle \\
\label {3902owdjcnjvhbyghfj4we78}  (m',n') \in \supp_{\gamma'}(F) \subseteq \langle (0,0),(m',0),(0,n'),(m',n') \rangle .
\end{gather}
We have $A=\bk[U,V]=\bk[X,Y]=R[X]$ where $R = \bk[Y]$.
Given $G \in A\setminus\{0\}$, write $G = \sum_{i=0}^d G_i X^i$ ($G_i \in R$, $G_d \neq0$)
and define $\deg_X(G)=d$ and $\lco(G) = G_d$.

Let $\alpha = \deg_X(U)$,  $\beta = \deg_X(V)$.
We claim that $\min(\alpha,\beta)=0$. To see this, assume that $\alpha,\beta > 0$.
Since $U,V$ are then variables of $\bk[X,Y]$ not belonging to $\bk[Y]$, we have $\lco(U), \lco(V) \in \bk^*$.
For each $(i,j) \in \supp_{\gamma'}(F)$ we have $i\le m'$ and $j\le n'$ by \eqref{3902owdjcnjvhbyghfj4we78}, so
$$
\deg_X( U^iV^j ) = \alpha i + \beta j \leq \alpha m' + \beta n' = \deg_X( U^{m'}V^{n'} )
$$
where equality holds iff $(i,j)=(m',n')$; so $\lco(F) = \lambda \lco(U^{m'}V^{n'})$ for some $\lambda \in \bk^*$
and consequently $\lco(F) \in \bk^*$. However, \eqref{7655u7fghbvkaks02} implies that $\lco(F)$ is a polynomial
in $\bk[Y]$ of degree $n\ge1$, a contradiction.
This shows that $\min(\deg_X(U),\deg_X(V))=0$ and we obtain $\min(\deg_Y(U),\deg_Y(V))=0$ by symmetry.
So, in (a), the left hand side is included in the right hand side.
As the reverse inclusion is trivial, (a) is proved.  Assertion (b) follows from (a).
\end{proof}

\begin{definition}  \label {89798xchh0c09x9cww9}
For each rectangular element $F$ of $A = \kk2$ we define 
$$
\text{$\bideg_A(F) = (\deg_X(F), \deg_Y(F))$ for any $(X,Y) \in \Rec^+(F,A)$.}
$$
By \ref{0cv9v8n8dllrklslk52kf0}, $\bideg_A(F)$ is well defined and depends only on $(F,A)$.\\
Moreover, if $(m,n) = \bideg_A(F)$ and $\gamma \in \Rec^+(F,A)$ then
$$
\text{$1 \le m \le n$\ \ and\ \  $(m,n) \in \supp_\gamma(F) \subseteq \langle (0,0),(m,0),(0,n),(m,n) \rangle$.}
$$
\end{definition}

\begin{remark}
Let $F$ be a rectangular element of $A = \kk2$ and let $(m,n) = \bideg_A(F)$.
It follows from \ref{0cv9v8n8dllrklslk52kf0}\eqref{9834fh87td4rth} that if $m=n$ then
$\Rec^+(F,A) = \Rec(F,A)$.
\end{remark}

We shall now consider the set $\Galg(F,A)$ defined in the introduction.
We first show that  $\Galg(F,A)$ is easy to describe when $F$ is a rectangular element of~$A$.

\begin{lemma}  \label {09239023r02n02b27c2c8h2}
Let $F$ be a rectangular element of $A = \kk2$, $\gamma = (X,Y) \in \Rec(F,A)$ and
$(m,n) = ( \deg_X(F), \deg_Y(F) )$. Recall that
$$
(m,n) \in \supp_\gamma(F) \subseteq \langle (0,0),(m,0),(0,n),(m,n) \rangle.
$$
Write $F = \sum_{i,j} a_{ij} X^iY^j$ ($a_{ij} \in \bk$) and define
$$
\textstyle
F_\text{\rm ver}(Y) = \sum_{j=0}^n a_{m,j}Y^j \quad \text{and} \quad
F_\text{\rm hor}(X) = \sum_{i=0}^m a_{i,n}X^j .
$$
\begin{enumerata}

\item \label {09293bbgxqweiby7sznxadfefj}
$\Galg(F,A)$ is equal to
$$
\setspec{ (X-a) }{ a\in\bk \text{ and } \deg F(a,Y)=1 }
\cup \setspec{ (Y-b) }{ b\in\bk \text{ and } \deg F(X,b)=1 }.
$$

\item \label {09dsfjk3346gbxbcdzrew24} If $\min(m,n)>1$ then $\Galg(F,A)$ is included in
$$
\setspec{ (X-a) }{ a\in\bk \text{ and } F_\text{\rm hor}(a)=0 }
\cup \setspec{ (Y-b) }{ b\in\bk \text{ and } F_\text{\rm ver}(b)=0 }.
$$

\end{enumerata}
\end{lemma}

\begin{proof}
Let $\pgoth \in \Galg(F,A)$.  As $A/\pgoth = \kk1$, we may consider a surjective
$\bk$-homomorphism $\pi : A \to \bk[t]=\kk1$ such that $\ker\pi = \pgoth$.
Let $x(t) = \pi(X)$ and $y(t) = \pi(Y)$.
Since the composite $\bk[F] \hookrightarrow A \xrightarrow{\pi} \bk[t]$ is an isomorphism
and $\pi(F) = F(x(t),y(t))$, we have $\deg_t  F(x(t),y(t)) = 1$.
If $x(t), y(t) \notin \bk$ then
$$
1 = \deg_t  F(x(t),y(t)) = m \deg_t x(t) + n \deg_t y(t) \ge m+n \ge 2,
$$
which is absurd; so $x(t) \in \bk$ or $y(t) \in \bk$, from which we obtain
$\pgoth = (Z-\lambda)$ for some $Z \in \{X,Y\}$ and $\lambda \in \bk$.
It is easily verified that $(X-\lambda) \in \Galg(F,A)$ $\Leftrightarrow$ $\deg F(\lambda,Y) = 1$
and that $(Y-\lambda) \in \Galg(F,A)$ $\Leftrightarrow$ $\deg F(X,\lambda) = 1$,
so (a) is proved.

Assume that $\min(m,n)>1$ and consider $(Z-\lambda) \in \Galg(F,A)$, 
where $Z \in \{X,Y\}$ and $\lambda \in \bk$.
If $Z=X$ then $\deg F(\lambda,Y)=1$; as $n>1$, it follows that $F_\text{\rm hor}(\lambda)=0$. 
Similarly, if $Z=Y$ then $F_\text{\rm ver}(\lambda)=0$. 
This proves (b).
\end{proof}

Note the following obvious consequence of \ref{09239023r02n02b27c2c8h2}:

\begin{corollary}  \label {0f29382c398y1hdq} 
Let $F$ be a rectangular element of $A = \kk2$.
Then each element of $\Galg(F,A)$ is of the form $(G)$ for some variable $G$ of $A$,
and each element of $\Gamma(F,A)$ is a coordinate line in $\aff^2 = \Spec A$.
\end{corollary}

\begin{remark}
If $G$ is any element of $A=\kk2$ such that $A/(G) = \kk1$
(note that $G$ is not necessarily a variable of $A$ if $\Char\bk>0$, see \ref{f02939d9wid29}) 
then there exists $F \in A \setminus \bk$ such that $(G) \in \Galg(F,A)$
(proof: write $A/(G) = \bk[t]$ and choose $F\in A$ whose image via $A \to A/(G)$ is $t$).
So \ref{0f29382c398y1hdq} states a nontrivial property of rectangular polynomials.
\end{remark}

We shall now ask how $\Gamma(F,A)$ behaves with respect to two operations:
extending the base field (\ref{653b34hj24h75f7hk676g})
and birationally extending the ring $A$ (\ref{56fewf8r8t34kd223lwgfuio}).

\begin{lemma}  \label  {653b34hj24h75f7hk676g}
Let $K/\bk$ be a field extension,
$A = \kk2$ and $\bar A = K \otimes_\bk A = K^{[2]}$.
Consider $F \in A \subseteq \bar A$ such that $F \notin \bk$. 
Then
\begin{equation}   \label {oiy5ksop8hri3hct}
\Galg(F,A) \to \Galg(F,\bar A),\quad \pgoth \mapsto \pgoth\bar A
\end{equation}
is a well defined injective map.
In particular, $| \Galg(F,A) | \leq | \Galg(F,\bar A) |$.
\end{lemma}

\begin{proof}
Consider $\pgoth \in \Galg(F,A)$ and write $\pgoth=GA$ where $G \in A$.
Then the composite $\bk[F] \to A \to A/GA$ is an isomorphism $\bk[F] \to A/GA$.
Applying the functor \mbox{$K \otimes_\bk\underline{\ \ }$} to $\bk[F] \to A \to A/GA$ yields
$K[F] \to \bar A \to \bar A/G\bar A$ where $K[F] \to \bar A/G\bar A$ is an isomorphism,
so $\pgoth\bar A = G\bar A$ is an element of $\Galg(F,\bar A)$ and
\eqref{oiy5ksop8hri3hct} is well defined.
Since $A \to \bar A$ is a faithfully flat homomorphism, we have $A \cap I \bar A = I$
for every ideal $I$ of $A$ (cf.\ \cite[(4.C)(ii), p.\ 27]{Matsumura}),
so \eqref{oiy5ksop8hri3hct} is injective.
\end{proof}

\begin{lemma} \label {56fewf8r8t34kd223lwgfuio}
Consider $F \in A \preceq A'$, where $F \notin\bk$.
Let $\Phi : \Spec A' \to \Spec A$ be the morphism determined by $A \hookrightarrow A'$.
\begin{enumerata}

\item For each $C' \in \Gamma(F,A')$, $\Phi(C')$ is a curve in $\Spec A$ and an element of $\Gamma(F,A)$.

\item The set map $\gamma : \Gamma(F,A') \to \Gamma(F,A)$, $C' \mapsto \Phi(C')$, is injective.

\item \label {324jhs7d5ghr2bv2npo24nc}
Let $C \in \Gamma(F,A)$. 
\begin{enumerata}

\item If $C' \subset \Spec A'$ is a curve such that $\Phi(C')=C$ and such that $\Phi|_{C'} : C' \to C$ is an isomorphism,
then $C' \in \Gamma(F,A')$ and $\gamma(C')=C$.

\item $C \in \image\gamma$ if and only if there exists
a curve $C' \subset \Spec A'$ such that $\Phi(C')=C$ and such that $\Phi|_{C'} : C' \to C$ is an isomorphism.

\end{enumerata}
\end{enumerata}
\end{lemma}

\begin{proof}
Consider an element $C'$ of $\Gamma(F,A')$ and let $C$ be the closure of $\Phi(C')$ in $\Spec A$.
As $C' \in \Gamma(F,A')$, 
we have morphisms $C' \to C \to \Spec\bk[F]$ whose composition is an isomorphism 
(in particular, $C$ is a curve).
Considering the coordinate rings $\bk[C]$ and $\bk[C']$ of $C$ and $C'$,
we have injective homomorphisms $\bk[F] \to \bk[C] \to \bk[C']$
such that the composition  $\bk[F] \to \bk[C']$ is an isomorphism; thus $\bk[F] = \bk[C] = \bk[C']$ and
hence $\Phi|_{C'} : C' \to C$ and $f|_{C} : C \to \Spec\bk[F]$ are two isomorphisms.
This implies that $\Phi(C')=C$ and $C \in \Gamma(F,A)$, which proves (a).
Assertion (b) follows from \ref{0923092cn3r98z29eu}\eqref{0as9a0f9s9s3487jhh3} and the
straightforward verification of (c) is left to the reader.
\end{proof}

\section{The cardinality of $\Gamma(F,A)$ for field generators}
\label {SEC:CardGammaFG}

The aim of this section is to prove Theorem~\ref{9854dnc2mrhvfc}, which asserts that $| \Gamma(F,A) | \le 2$
whenever $F$ is a field generator of $A=\kk2$ which is not of the form $\alpha(Y)X + \beta(Y)$.
In the course of proving that result, we obtain Theorem~\ref{8923gdtswA2y87oa9fhcbnccv},
which gives new information on ``small'' field generators.

\medskip
We shall make essential use of two results of Russell on field generators.
The first one (\ref{982393298urcnj092}) is valid over an arbitrary field $\bk$:

\begin{theorem}[{\cite[3.7 and 4.5]{Rus:fg}}]  \label {982393298urcnj092}
If $F$ is a field generator of $A=\kk2$ which is not a variable of $A$, then $F$ is a rectangular element of $A$.
\end{theorem}

In view of \ref{982393298urcnj092} and of the objective of this section,
it is interesting to note that there is no upper bound on the cardinality of $\Gamma(F,A)$
for rectangular elements $F$ of $A=\kk2$:

\begin{examples}  \label {8923ry7fhuh}
\begin{enumerata}

\item Let $F = u(Y)X^2 + X \in A = \bk[X,Y]$ where $\deg u(Y) > 1$;
then \ref{09239023r02n02b27c2c8h2} implies that $|\Gamma(F,A)|$ equals the number of roots of $u(Y)$.

\item \label {923hvmnbwkuwq5}
Assume that $\bk$ is infinite and let $F = \alpha(Y)X + \beta(Y) \in A = \bk[X,Y]$
where $\deg\beta(Y) \le \deg \alpha(Y) > 0$;
then \ref{09239023r02n02b27c2c8h2} implies that $|\Gamma(F,A)| = |\bk|$.
Note that $F$ is a good field generator of $A$ since $\bk(F,Y)=\bk(X,Y)$.

\end{enumerata}
\end{examples}

The other result of Russell that we need is  \cite[1.6]{Rus:fg2}.
Because there seems to be an ambiguity in the statement of that result, let us explain the following.
(The statement of that result is too long to be reproduced here.)
Suppose that $F(X,Y) \in \bk[X,Y]$ satisfies all hypotheses of that theorem (in particular $\mu_1 \le \mu_2$).
If $\mu_1 < \mu_2$ then the statement is clear,
and all conclusions of \cite[1.6]{Rus:fg2} are true for $F(X,Y)$.
If $\mu_1=\mu_2$ then both $F(X,Y)$ and $F(Y,X)$ satisfy the hypotheses of the theorem,
but the proof only shows that the theorem is true {\it for at least one of these polynomials}.
(In the proof, just after (5) on page 320 of \cite{Rus:fg2}, one reads 
``It follows that $h_1=1$ or $l_1=1$. Say $h_1=1$''.
The intended meaning of that sentence is: ``Replacing $F(X,Y)$ by $F(Y,X)$ if necessary, we may assume
that $h_1=1$''.)
There are indeed examples with $\mu_1=\mu_2$ where the theorem is false for $F(X,Y)$ and true for $F(Y,X)$.
We shall use  \cite[1.6]{Rus:fg2} in the proofs of 
\ref{pf900293012dj0mce19}\eqref{0c987w53xhgr3jrqh} and \ref{8923gdtswA2y87oa9fhcbnccv};
the ambiguity arises in the first proof only.

\begin{definition}
Let $F \in A = \kk2$ and let $\gamma = (X,Y)$ be such that $A = \bk[X,Y]$.
We say that $F$ is {\it $\gamma$-small\/} in $A$ if the following conditions are satisfied:
\begin{itemize}

\item $F \notin \bk[X,u(X)Y]$ for all $u(X) \in \bk[X]\setminus \bk$;
\item $F \notin \bk[Xv(Y),Y]$ for all $v(Y) \in \bk[Y]\setminus \bk$.

\end{itemize}
\end{definition}

\begin{remark} \label {3wfhjjos3dhjjfio}
It follows from \ref{0cv9v8n8dllrklslk52kf0} that,
for a rectangular element $F$ of $A = \kk2$, the following are equivalent:
\begin{itemize}

\item $F$ is $\gamma$-small in $A$ for at least one $\gamma \in \Rec(F,A)$;
\item $F$ is $\gamma$-small in $A$ for all $\gamma \in \Rec(F,A)$.

\end{itemize}
\end{remark}

\begin{definition}
By a {\it small\/} field generator of $A = \kk2$, we mean a field generator $F$ of $A$ 
for which there exists $\gamma\in \Rec(F,A)$ such that $F$ is $\gamma$-small in $A$.
(Then, by \ref{3wfhjjos3dhjjfio}, $F$ is $\gamma$-small in $A$ for every $\gamma\in \Rec(F,A)$.)
\end{definition}

\begin{lemma}  \label {pf900293012dj0mce19} 
If $F$ is a small field generator of $A = \kk2$ then the following hold.
\begin{enumerata}

\item \label {27364dtrfg}
$\Rec(F,A) \neq \emptyset$ and, for all $\gamma\in \Rec(F,A)$, $F$ is $\gamma$-small in $A$.

\item $F$ is not a variable of $A$.

\item \label {712tgxuydc82}
The pair $(m,n) = \bideg_A(F)$ (cf.\ \ref{89798xchh0c09x9cww9}) satisfies $2 \le m \le n$.

\item \label {0c987w53xhgr3jrqh} 
If $\bk$ is algebraically closed then $m<n$.

\end{enumerata}
\end{lemma}

\begin{proof}
(a) $\Rec(F,A) \neq \emptyset$ is clear and the other assertion is \ref{3wfhjjos3dhjjfio}.
Since no variable of $A$ is rectangular, (b) is true.
(c) Pick $\gamma = (X,Y) \in \Rec^+(F,A)$; then $F$ is $\gamma$-small and $(m,n) = ( \deg_X(F), \deg_Y(F) )$.
We have $1 \le m \le n$ by definition.
If $m=1$ then $F = a(Y)X+b(Y)$ for some $a(Y), b(Y) \in \bk[Y]$, $a(Y) \notin \bk$;
then $F \in \bk[a(Y)X,Y]$, which contradicts the fact that $F$ is $\gamma$-small.

(d) Proceeding by contradiction, assume that $m=n$.
Pick $\gamma=(X,Y) \in \Rec^+(F,A)$, note that $A=\bk[X,Y]$ and view $F\in A$ as a polynomial in $X,Y$.
Then $F(X,Y) \in \bk[X,Y]$ satisfies the hypothesis of \cite[1.6]{Rus:fg2}
(with $\mu_1=m$ and $\mu_2=n$, so $\mu_1=\mu_2$), and so does $F(Y,X)$.
By the discussion just after \ref{8923ry7fhuh}, there exists $G(X,Y) \in \{ F(X,Y), F(Y,X) \}$ such that
the conclusions of \cite[1.6]{Rus:fg2} are true for $G(X,Y)$.
So, by assertion (6) of that result, 
there exists $c \in \bk$ such that $G(X,Y) \in \bk[X,(X-c)Y]$
(choosing such a $c\in\bk$ is equivalent to the choice made in assertion (6)).
It follows that $F(X,Y)$ belongs to $\bk[X,(X-c)Y]$ or $\bk[X(Y-c), Y]$,
so $F$ is not $\gamma$-small.
This contradicts the hypothesis that $F$ is  a small field generator of~$A$.
\end{proof}

The next paragraph introduces notation that we need for proving Theorem~\ref{8923gdtswA2y87oa9fhcbnccv}.

\begin{parag} \label {9238vvx761b71wtugs}
Let $\bk$ be an algebraically closed field, 
$F = \sum_{i,j} a_{ij}X^iY^j \in \bk[X,Y]$ ($a_{ij} \in \bk$) a polynomial of degree $d>0$ and
$F^*(W,X,Y) = \sum_{i,j} a_{ij}W^{d-i-j}X^iY^j \in \bk[W,X,Y]$ the homogenization of $F$.
Consider the pencil $\Lambda = \Lambda(F)$ on $\proj^2$ defined by
$$
\Lambda = \Lambda(F) = \setspec{ \div_0( \lambda_0 F^* - \lambda_1 W^{d} ) }{ (\lambda_0 : \lambda_1) \in \proj^1 },
$$
where ``$\div_0$'' means ``divisor of zeroes'' (of a homogeneous polynomial).
Note that $dE_0 \in \Lambda$, where $E_0 = V(W)$ is the line at infinity.
Let $\Beul$ be the set of base points of $\Lambda$, including infinitely near ones, and note that $\Beul$ is a finite set.
Consider the minimal resolution of the base points of $\Lambda$,
$$
S_n \xrightarrow{\pi_n} S_{n-1} \to \cdots \to S_1 \xrightarrow{\pi_1} S_0 = \proj^2,
$$
where $S_0, \dots, S_n$ are nonsingular projective surfaces and,
for each $i \in \{1, \dots, n\}$,
$\pi_i : S_i \to S_{i-1}$ is the blowing-up of $S_{i-1}$ at the point $P_{i} \in S_{i-1}$
(so $\Beul = \{P_1, \dots, P_n\}$).
Write $E_i = \pi_i^{-1}(P_i) \subset S_i$ for $i \in \{1, \dots, n\}$.
Given a point $P$ and a divisor $D$ on some nonsingular surface, the multiplicity of $P$ on $D$
is denoted $\mu(P,D)$.
For each $i \in \{0, \dots, n\}$, let $\Lambda^{(i)}$ be the strict (or proper) transform of $\Lambda$ on $S_{i}$
(defined in \cite{Rus:fg}, just before 2.4); note that $\Lambda^{(0)}=\Lambda$.
\begin{enumerata}

\item Let $i \in \{1, \dots, n\}$.
Then $\Lambda^{(i-1)}$ is a pencil on $S_{i-1}$ and $P_{i} \in S_{i-1}$ is a base point of it.
The positive integer $\inf\setspec{ \mu(P_{i},D) }{ D \in \Lambda^{(i-1)} }$
is called the multiplicity of $P_i$ as a base point of $\Lambda$, and is denoted $\mu(P_i,\Lambda)$.
We shall abbreviate it $\mu(P_i)$ in the proof of \ref{8923gdtswA2y87oa9fhcbnccv}.

\item \label {998vgcb1vdfdbzxxbenxbdl}
Let $P$ be a point of $S_i$ and $D$ a divisor on $S_j$ 
($i,j \in \{0, \dots, n\}$).  Define $\mu(P,D) \in \Integ$ by
$$
\mu(P, D ) = \begin{cases}
0, & \text{if $i < j$}, \\
\mu(P,\tilde D) \text{ where $\tilde D \subset S_{i}$ is the strict transform of $D$}, & \text{if $i \ge j$}.
\end{cases}
$$
So $\mu(P_i,E_j) \in \{0,1\}$ for all $i,j \in \{1, \dots, n\}$. 

\item \label {9dfu32i621fcznxj1ny2r}
Let $\Lambda_\infty^{(n)}$ be the unique element of $\Lambda^{(n)}$ whose support contains the strict
transform of $E_0$ (so $\Lambda_\infty^{(n)}$ is an effective divisor on $S_n$).
Given $i \in \{0, \dots, n\}$ and an irreducible curve $D \subset S_i$,
let $\epsilon(D) \in \Nat$ be the coefficient of $\tilde D$ in the divisor $\Lambda_\infty^{(n)}$,
where $\tilde D \subset S_n$ denotes the strict transform of $D$
($\epsilon(D)$ is also defined in \cite[3.4]{Rus:fg}).

\end{enumerata}
\end{parag}

We shall now improve \ref{982393298urcnj092} in the special case where $F$ is a small field generator
and $\bk$ is algebraically closed:

\begin{theorem}  \label {8923gdtswA2y87oa9fhcbnccv}
Assume that $\bk$ is algebraically closed and let $F$ be a small field generator of $A=\kk2$.
Then there exist $\gamma = (X,Y) \in \Rec^+(F)$ and $(m,n) \in \Nat^2$ satisfying 
\begin{equation}  \tag{$*$}
(m,n) \in \supp_\gamma(F) \subseteq \langle (0,0), (m,0), (0,n-m), (m,n) \rangle  \text{\ \ and\ \ }
(m,0) \notin \supp_\gamma(F) .
\end{equation}
Moreover, for any such $\gamma = (X,Y)$ and $(m,n)$, the following hold.
\begin{enumerata}

\item \label {a-65b76rnyrhjdsjddvn} $(m,n) = \bideg_A(F)$ and $2 \le m < n$.

\item \label {b-65b76rnyrhjdsjddvn} If we write $F = \sum_{i,j} a_{ij}X^iY^j$ ($a_{ij} \in \bk$) then
the polynomial $H(Y) = \sum_{j=0}^n a_{mj} Y^j$ has either one or two roots.

\item \label {c-65b76rnyrhjdsjddvn} If $H(Y)$ has one root then
$\supp_\gamma(F) \subseteq \langle (0,0), (m-1,0), (0,n-m), (m,n) \rangle$
and $\Galg(F,A) \subseteq \{ (X), (Y) \}$.

\item \label {d-65b76rnyrhjdsjddvn} If $H(Y)$ has two roots then the following hold.
\begin{enumerata}

\item $m \mid n$

\item  Consider the homogenization $F^*(W,X,Y) = \sum_{i,j} a_{ij} W^{m+n-i-j}X^iY^j \in \bk[W,X,Y]$ of $F$
and note that $(0:0:1) \in \proj^2$ is a common point of the curve $V(F^*) \subset \proj^2$ and
of the line at infinity $V(W)$.
Then exactly one element $R$ of $\bbV^\infty(F,A)$ is centered at~$(0:0:1)$, and $[R/\mgoth_R : \bk(F)] = m$
(cf.\ \ref{939ri09109e2xj9cru}).
That is, $F$ has exactly one dicritical over the point $(0:0:1)$ and that dicritical has degree $m$.

\item Let $k = \deg F(0,Y)$ (where $k=0$ if $F(0,Y)=0$).
Then $m \mid k$, $k \leq n-m$ and 
$\supp_\gamma(F) \subseteq \langle (0,0), (m,0), (0,k), (m,n) \rangle$.

\item $\Galg(F,A) \subseteq \{ (Y-r_1), (Y-r_2) \}$, where $r_1,r_2$ are the roots of $H(Y)$.

\end{enumerata}

\end{enumerata}
\end{theorem}

The above theorem immediately implies the following (to be improved in \ref{9854dnc2mrhvfc}):

\begin{corollary}  \label {fj29o9232g7qwjda}
Assume that $\bk$ is algebraically closed and let $F$ be a small field generator of $A=\kk2$.
Then $| \Gamma(F,A) | \le 2$.
\end{corollary}

\begin{proof}[Proof of \ref{8923gdtswA2y87oa9fhcbnccv}]
Since $F$ is a small field generator of $A$, $\Rec^+(F,A) \neq \emptyset$.
So $(m,n) = \bideg_A(F)$ is defined and $2 \le m < n$ by \ref{pf900293012dj0mce19}.
For any choice of $\gamma_1=(X_1,Y_1) \in \Rec^+(F,A)$, we have
$ (m,n) \in \supp_{\gamma_1}(F) \subseteq \langle (0,0), (m,0), (0,n), (m,n) \rangle$.
The pair $\gamma_1$ being given, there exist $a,b \in \bk$ such that the element
$\gamma=(X,Y)=(X_1-a,Y_1-b)$ of $\Rec^+(F,A)$ satisfies
$(m,0), (0,n) \notin \supp_\gamma(F)$ (we use that $\bk$ is algebraically closed here).
So there exists $\gamma=(X,Y) \in \Rec^+(F,A)$ satisfying
\begin{equation}  \label {8293dh81hxnbcbrxxmnrhc3y}
\text{\small
$(m,n) \in \supp_\gamma(F) \subseteq \langle (0,0), (m,0), (0,n), (m,n) \rangle  \text{\ \ and\ \ }
(m,0),(0,n) \notin \supp_\gamma(F)$},
\end{equation}
and we fix such a $\gamma$ from now-on.
We shall prove that $(*)$ and (\ref{a-65b76rnyrhjdsjddvn}--\ref{d-65b76rnyrhjdsjddvn}) hold.
In fact we already noted that \eqref{a-65b76rnyrhjdsjddvn} is true.

Consider the pencil $\Lambda = \Lambda(F)$ on $\proj^2$ (cf.\ \ref{9238vvx761b71wtugs}).
Then $\Lambda$ has two base points on $\proj^2$, namely, $p_0 = (0:0:1)$ and $q_0=(0:1:0)$.
Let $\Beul$ be the (finite) set of all base points of $\Lambda$, including infinitely near ones.
Define a partial order $\leq$ on the set $\Beul$ by declaring that $q < q'$ $\Leftrightarrow$ $q'$ is infinitely near $q$;
note that $p_0$ and $q_0$ are exactly the minimal elements of the poset $(\Beul,\leq)$.
For each $q \in \Beul$,
let $\mu(q,\Lambda)$ denote the multiplicity of $q$ as a base point of $\Lambda$ (cf.\ \ref{9238vvx761b71wtugs})
and let us use the abbreviation $\mu(q) = \mu(q,\Lambda)$.
Given $q \in \Beul$, let $E_q$ denote the exceptional curve created by blowing-up $q$
(in the notation of \ref{9238vvx761b71wtugs}, if $q=P_i$ then $E_q=E_i$).
Given $p,q \in \Beul$, $\mu(p,E_q) \in \{0,1\}$ is defined in
\ref{9238vvx761b71wtugs}\eqref{998vgcb1vdfdbzxxbenxbdl}.
We also define
\begin{equation} \label {92378tr78grxfemnnzvxc}
\text{$\Beul_q = \setspec{ p \in \Beul }{ \mu(p,E_{q})>0 }$ for each $q \in \Beul$.}
\end{equation}

Write $F = \sum_{i,j} a_{ij}X^iY^j$ ($a_{ij} \in \bk$) and
$F^*(W,X,Y) = \sum_{i,j} a_{ij}W^{m+n-i-j}X^iY^j$ as in the statement.
Define $F_{p_0}(W,X) = F^*(W,X,1)$ and $F_{q_0}(W,Y) = F^*(W,1,Y)$; then
\begin{align}
F_{p_0}(W,X) &
= \textstyle \sum_{i,j} a_{ij}W^{m+n-i-j}X^i
= \sum_{i=0}^m a_{i,n}W^{m-i}X^i \ +\ \text{higher order terms}, \\
\label {98439dxfgghyjdnh}
F_{q_0}(W,Y) & = \textstyle \sum_{i,j} a_{ij}W^{m+n-i-j}Y^j
= \sum_{j=0}^n a_{m,j}W^{n-j}Y^j \ +\ \text{higher order terms}
\end{align}
where $a_{m,n} \neq0$, so $\mu(p_0) = m < n = \mu(q_0)$.
Let us define $\mu_1=m$ and $\mu_2=n$; then our notation is compatible with  \cite[1.6]{Rus:fg2}
and the hypothesis of that result is satisfied.
Since $\mu_1<\mu_2$, the ambiguity noted just after \ref{8923ry7fhuh} does not arise here,
so assertions \mbox{\rm (1--6)} of \cite[1.6]{Rus:fg2} are true for $F$.
By part~(2) of that result, there is a unique integer $s\ge1$ such that (i)~there are $s+1$
base points $p_0 < \cdots < p_s$ i.n.\ to $p_0$ with $\mu(p_i)=\mu_1$ for $i=0,\dots,s$,
and (ii)~$|\Beul_{p_s}| \neq 1$ (see \eqref{92378tr78grxfemnnzvxc}).
This defines $s \ge1$ and $p_0,p_1,\dots,p_s$.
By \cite[1.6(3)]{Rus:fg2}, if we define  $\nu = \mu_2 - s\mu_1$ then  
(since $F$ is small by assumption) $0 \le \nu < \mu_1$. That is,
\begin{equation}  \label  {iuxmsm32c4hi7t}
\mu_2 = s\mu_1 + \nu \text{\ \ and\ \ }  0 \le \nu < \mu_1 < \mu_2.
\end{equation}
Since $s\ge1$, there holds $\mu(p_0)=\mu(p_1)=m$;
so, if we write $F_{p_0}(W,X) = \sum_{u,v} b_{u,v}W^{u}X^v$ ($b_{u,v} \in \bk$),
taking into account that $(0,n) \notin \supp_\gamma(F)$
we obtain $u+2v \ge 2m$ for all $(u,v)$ satisfying $b_{u,v} \neq 0$;
it follows that $j-i \le n-m$ for all $(i,j) \in \supp_\gamma(F)$.
In view of \eqref{8293dh81hxnbcbrxxmnrhc3y}, we conclude that condition $(*)$ is satisfied.
There remains to prove (\ref{b-65b76rnyrhjdsjddvn}--\ref{d-65b76rnyrhjdsjddvn}).

\medskip
Let us write $\Beul_{p_s} =  \{ p_{11}, \dots, p_{1h} \}$ and $\Beul_{q_0} =  \{ q_{11}, \dots, q_{1\ell} \}$
(see \eqref{92378tr78grxfemnnzvxc}); this defines $h,\ell,p_{1j},q_{1j}$ in a way that is completely
compatible with  \cite[1.6]{Rus:fg2}.
Since $F$ is small, 
\begin{equation}  \label {67er5cwzs413ws2az}
\ell = s+1
\end{equation}
by  \cite[1.6(5)]{Rus:fg2},
so in particular $\ell>0$; however, $h$ may be zero.
Let us also set $\Beul' = \Beul \setminus \big( 
\{ p_0, \dots, p_s \} \cup  \{ p_{11}, \dots, p_{1h} \} \cup \{ q_0 \} \cup  \{ q_{11}, \dots, q_{1\ell} \} \big)$.
Then we have the disjoint union
\begin{equation}  \label {92523cnfgf3dd}
\Beul
=  \{ p_0, \dots, p_s \}
\cup  \{ p_{11}, \dots, p_{1h} \}
\cup \{ q_0 \}
\cup  \{ q_{11}, \dots, q_{1\ell} \}
\cup \Beul' .
\end{equation}

We claim:
\begin{equation}  \label {0f298e28ge87fs4}
\begin{minipage}[t]{.9\textwidth}
$H(Y)$ has either one or two roots, and if it has two then $h=0=\nu$.
\end{minipage}
\end{equation}
To see this, consider the number  $\epsilon(E_{q_0})$ defined in \ref{9238vvx761b71wtugs}\eqref{9dfu32i621fcznxj1ny2r}.
By 3.5.2 and 3.5.4 of \cite{Rus:fg}, $\epsilon(E_{q_0}) = \epsilon(E_0)-\mu(q_0) = d - \mu_2$,
where $E_0$ is the line $\proj^2 \setminus \aff^2$ and $d = \deg(F) = \mu_1+\mu_2$, so
\begin{equation}  \label {78cfsw27kql2kd}
\epsilon(E_{q_0}) = \mu_1 > 0.
\end{equation}
Let $\Lambda^{(q_0)}$ denote the strict transform of $\Lambda$ with respect to the blowing-up of $\proj^2$ at $q_0$.
Since $\epsilon(E_{q_0}) > 0$,  \cite[3.5.6]{Rus:fg} implies that
\begin{equation} \label {9928127dd21dfgdfihfw}
\text{for a general member $D$ of $\Lambda^{(q_0)}$, all points of  $D \cap E_{q_0}$ belong to $\Beul$.}
\end{equation}

Let $\rho$ be the number of distinct roots of $H(Y)$.
Then the initial form $\sum_{j=0}^n a_{m,j}W^{n-j}Y^j$ of $F_{q_0}(W,Y)$ is a product
$\prod_{i=1}^\rho L_i^{e_i}$ where $L_1, \dots, L_\rho \in \bk[W,Y]$ are pairwise relatively prime 
linear forms and $e_i\ge1$ for all $i = 1, \dots, \rho$.
So, for a general member $D$ of $\Lambda^{(q_0)}$, $D \cap E_{q_0}$ consists of exactly $\rho$ points.
In view of \eqref{9928127dd21dfgdfihfw}, these $\rho$ points belong to $\Beul$; in fact they must be the
minimal elements of $\{ q_{11}, \dots, q_{1\ell} \}$ with respect to the order relation
of the poset $(\Beul,\leq)$.
So $\{ q_{11}, \dots, q_{1\ell} \}$ has exactly $\rho$ minimal elements.
Consequently, to  prove \eqref{0f298e28ge87fs4} it suffices to show that
\begin{equation}  \label {912gfxvnrn4nqtyuq7n}
\text{$\{ q_{11}, \dots, q_{1\ell} \}$ has $1$ or $2$ minimal elements, and
if it has $2$ then $h=0=\nu$.}
\end{equation}

Let us state the facts that we need for proving \eqref{912gfxvnrn4nqtyuq7n}.
We use the abbreviation $\mu(\Seul) = \sum_{p \in \Seul} \mu(p)$ for any subset $\Seul$ of $\Beul$.

\begin{enumerate}
\renewcommand{\theenumi}{\roman{enumi}}

\item \label {i-3r928hdc}  $\mu(\Beul_q) = \mu(q)$ for all $q \in \Beul$ satisfying $\epsilon(E_{q}) > 0$.

\item \label {ii-3r928hdc}  $\mu(q)<\mu_1$ for all $q \in \Beul \setminus \{ p_0, \dots, p_s, q_0 \}$.

\item \label {iii-3r928hdc}  $\epsilon(E_{q}) > 0$ for all $q \in \{ q_{11}, \dots, q_{1\ell} \}$.

\item \label {viii-3r928hdc}  $\mu( p_{i} ) = \mu_1$ for $0 \le i \le s$.

\item \label {iv-3r928hdc}  $\sum_{i=1}^h \mu( p_{1i} ) = \delta \mu_1$ where we define $\delta = \begin{cases}
0 & \text{if $h=0$,}\\
1 & \text{if $h\neq0$.}
\end{cases}$

\item \label {v-3r928hdc}  $\mu(q_0)=\mu_2$ and $\sum_{i=1}^\ell \mu( q_{1i} ) = \mu_2$

\item \label {vi-3r928hdc}  $\mu(\Beul) = 3d-2$, where $d = \deg(F) = \mu_1+\mu_2$.

\end{enumerate}

Proof of (\ref{i-3r928hdc}--\ref{vi-3r928hdc}).
Assertion \eqref{i-3r928hdc} is a well-known consequence of the intersection formula.
Since $F$ is small, \cite[1.6(5)]{Rus:fg2} implies that
$\mu(q)<\mu_1$ for all  $q \in \{ q_{11}, \dots, q_{1\ell} \}$; \eqref{ii-3r928hdc} easily follows.
Given $q \in \{ q_{11}, \dots, q_{1\ell} \}$ we have $\mu(q,E_{q_0})>0$, so 
\cite[3.5.4]{Rus:fg} gives $\epsilon(E_q) \ge \epsilon(E_{q_0}) - \mu(q)$; then
$\epsilon(E_q) \ge  \mu_1-\mu(q) > 0$ by \eqref{78cfsw27kql2kd} and \eqref{ii-3r928hdc}, proving \eqref{iii-3r928hdc}.
\eqref{viii-3r928hdc} follows from the definition of $p_0,\dots,p_s$. 
\eqref{iv-3r928hdc} is obvious if $h=0$, and follows from either one of \eqref{i-3r928hdc} or \cite[1.6(4)]{Rus:fg2}
if $h \neq0$.
We already noted that $\mu(q_0) = n = \mu_2$; the second part of \eqref{v-3r928hdc} then
follows from either one of \eqref{i-3r928hdc} or \cite[1.6(5)]{Rus:fg2}. 
\eqref{vi-3r928hdc} follows from \cite[3.3]{Rus:fg} with $g=0$.
This proves (\ref{i-3r928hdc}--\ref{vi-3r928hdc}).

Since \eqref{92523cnfgf3dd} is a disjoint union,
\begin{multline*}
\mu(\Beul)
= \sum_{i=0}^s \mu(p_i)
+  \sum_{i=1}^h \mu(p_{1i})
+ \mu( q_0 )
+  \sum_{i=0}^\ell \mu(q_{1i})
+ \mu( \Beul' )
= (s+1)\mu_1 + \delta\mu_1 + 2 \mu_2 + \mu(\Beul') .
\end{multline*}
On the other hand, $\mu( \Beul ) = 3d-2$, $d = \mu_1+\mu_2$ and $\mu_2 = s\mu_1+\nu$; so
\begin{equation} \label {745nxvhcuhgdiu823}
\mu(\Beul') = (2-\delta)\mu_1 + \nu - 2.
\end{equation}
Let $\Meul$ be the set of maximal elements of $\{ q_{11}, \dots, q_{1\ell} \}$ and let
$\Neul = \{ q_{11}, \dots, q_{1\ell} \} \setminus \Meul$.
For each $q \in \Meul$, we have $\epsilon(E_q)>0$ by \eqref{iii-3r928hdc},
so $\mu(q) = \mu( \Beul_q )$ by \eqref{i-3r928hdc}.
Moreover, $\Beul_q \subseteq \Beul'$ for all $q \in \Meul$
and $\Beul_q \cap \Beul_{q'} = \emptyset$ for all choices of distinct $q,q' \in \Meul$.
Thus
$$
\mu(\Beul') \ge \mu( \bigcup_{q \in \Meul} \Beul_q )
=\sum_{q \in \Meul} \mu(\Beul_q)
=\sum_{q \in \Meul} \mu(q)
=\mu( \Meul )
$$
and combining this with \eqref{745nxvhcuhgdiu823} gives
\begin{equation} \label {dddssydtg924893758346}
\mu(\Meul) \le  (2-\delta)\mu_1 + \nu - 2.
\end{equation}
We have $\mu(q) \le \mu_1$ for all $q \in \Neul$ by \eqref{ii-3r928hdc}  
and $|\Neul| = \ell - |\Meul| = s+1-|\Meul|$ by \eqref{67er5cwzs413ws2az}, so
\begin{equation} \label {34hv9wjkeh5dso12k}
\mu(\Neul) \le (s+1-|\Meul|)\mu_1.
\end{equation}
By \eqref{dddssydtg924893758346} and \eqref{34hv9wjkeh5dso12k},
\begin{multline*}
s\mu_1+\nu = \mu_2 = \mu( \{ q_{11}, \dots, q_{1\ell} \} ) = \mu(\Meul) + \mu(\Neul) \\
\le (2-\delta)\mu_1 + \nu - 2 + (s+1-|\Meul|)\mu_1
\end{multline*}
so $( |\Meul| + \delta - 3) \mu_1 \leq -2$ and consequently
\begin{equation} \label {23fhbmjoi98654sx}
|\Meul|+\delta \leq 2 .
\end{equation}
Let $\rho$ be the number of minimal elements of $\Beul_{q_0} = \{ q_{11}, \dots, q_{1\ell} \}$.
Then it is clear that $\rho \le |\Meul|$ (actually, $\rho = |\Meul|$ but we don't need to know this),
so $\rho + \delta \le2$ by \eqref{23fhbmjoi98654sx}.
So $\rho$ is either $1$ or $2$, and if it is $2$ then $\delta=0$,
so $h=0$, so $\nu=0$ (for the fact that $h=0$ implies $\nu=0$, see the line just after (8) on page 321 of \cite{Rus:fg2}).
This proves \eqref{912gfxvnrn4nqtyuq7n} and hence \eqref{0f298e28ge87fs4}.
In particular, assertion \eqref{b-65b76rnyrhjdsjddvn} is proved.

\medskip
If $H(Y)$ has one root then, since we arranged that $(m,0) \notin \supp_\gamma(F)$, $H(Y) = a_{mn}Y^n$.
This proves the assertion about $\supp_\gamma(F)$, in \eqref{c-65b76rnyrhjdsjddvn}.
The assertion about $\Galg(F,A)$ then follows from  \ref{09239023r02n02b27c2c8h2}\eqref{09dsfjk3346gbxbcdzrew24}.
So \eqref{c-65b76rnyrhjdsjddvn} is proved.

\medskip
To prove (d), consider the diagram
\begin{equation}  \label {2-dkfjasodfla2}
\xymatrix{
\aff^2 \ar[d]_{f} \ar @<-.4ex> @{^{(}->}[r] & \proj^2  \ar @{-->} [d]_{\phi_\Lambda}  &  X \ar[dl]^{\bar f} \ar[l]_{\pi}  \\
\aff^1 \ar @<-.4ex> @{^{(}->}[r] & \proj^1
}
\end{equation}
where $f$ is the morphism $\Spec A \to \Spec \bk[F]$,
$\phi_\Lambda : \proj^2 \dasharrow \proj^1$ is the rational map determined by $f$
(with domain $\proj^2 \setminus \{p_0,q_0\}$),
$\pi$ is the blowing-up of $\proj^2$ along $\Beul$ and $\bar f$ is a morphism; this gives rise to a
diagram \eqref{dkfjasodfla2}, which (as explained in \ref{difqwkae.dmnkdfuuuu6}) allows us to identify 
the set $\bbV^\infty(F,A)$ of dicriticals with the set $H^\infty$ of horizontal curves at infinity.

Assume that $H(Y)$ has two roots, $r_1,r_2 \in \bk$. Then $h=\nu=0$ by \eqref{0f298e28ge87fs4}.
Since $\nu=0$, \eqref{iuxmsm32c4hi7t} implies that $m \mid n$, so (\ref{d-65b76rnyrhjdsjddvn}-i) is true.
The fact that $h=0$ implies that $\setspec{ q \in \Beul }{ q \ge p_0 } = \{ p_0, p_1, \dots, p_s \}$;
since $\mu(p_i) = \mu_1$ for all $i = 0, \dots, s$,
we see that $\tilde E_{p_s}$ (the strict transform of $E_{p_s}$ on $X$) is the only element $C \in H^\infty$
satisfying $\pi(C) = \{ p_0 \}$; it follows that (\ref{d-65b76rnyrhjdsjddvn}-ii) is true.

To prove (\ref{d-65b76rnyrhjdsjddvn}-iii), we consider the elements
$(m,0)$ and $(0,m+n-k)$ of the support of $F_{p_0}(W,X)$ and the line segment
$L \subset \Reals^2$ joining those two points.
It follows from (\ref{d-65b76rnyrhjdsjddvn}-ii) that $L$ is an edge of the support of $F_{p_0}(W,X)$
and that $m$ divides $m+n-k$.
This implies that $m \mid k$ and that 
$\supp_\gamma(F) \subseteq \langle (0,0), (m,0), (0,k), (m,n) \rangle$,
so (\ref{d-65b76rnyrhjdsjddvn}-iii) is true.

By \ref{09239023r02n02b27c2c8h2}\eqref{09dsfjk3346gbxbcdzrew24}, we have 
$\Galg(F,A) \subseteq \{ (X), (Y-r_1), (Y-r_2) \}$.
We have $(X) \notin \Galg(F,A)$ because $\deg F(0,Y)$ is equal to the integer $k$ of (\ref{d-65b76rnyrhjdsjddvn}-iii),
and $k \neq1$. So  (\ref{d-65b76rnyrhjdsjddvn}-iv) is true.
\end{proof}

\begin{lemma}  \label {qn03h7tg5rvbtc07yr}
Let $\bk$ be an algebraically closed field and let $F$ be a rectangular element of $A=\kk2$ satisfying
$| \Galg(F,A) | \ge 2$.
Suppose that $(X,Y) \in \Rec(F,A)$ and $v(Y) \in \bk[Y] \setminus \bk$ are such that
$F$ is a variable of $A' = \bk[ Xv(Y), Y ]$.
Then $F = c Xv(Y)+w(Y)$ for some $c \in \bk^*$ and $w(Y) \in \bk[Y]$ such that $\deg w(Y) \le \deg v(Y)$.
\end{lemma}

\begin{proof}
We may write $F(X,Y)=G(Xv(Y),Y)$ where $G(S,T)$ is a variable of $\bk[S,T] =\kk2$.
Since $F$ is rectangular, we have $F \notin \bk[Y]$, so $G(S,T) \notin \bk[T]$.
It follows that $\deg_S(G) \ge1$, and it suffices to show that $\deg_S(G)=1$.
Proceeding by contradiction, we assume that 
$$
\deg_S(G) > 1 .
$$
Since $| \Galg(F,A) | \ge 2$ and $(X,Y) \in \Rec(F,A)$,
\ref{09239023r02n02b27c2c8h2}\eqref{09293bbgxqweiby7sznxadfefj} implies that
at least one of the following conditions holds:
\begin{enumerate}

\item[(i)] There exists $b \in \bk$ such that $(Y-b) \in \Galg(F,A)$;

\item[(ii)] there exist distinct $a_1,a_2 \in \bk$ such that $(X-a_1), (X-a_2) \in \Galg(F,A)$.

\end{enumerate}

Suppose that (i) holds.
Write $G(S,T) = \sum_{i=0}^m G_i(T)S^i$ where $G_m(T) \neq 0$.
Then $m = \deg_S(G) \ge 2$; by a well-known property of variables, $G_m(T) \in \bk^*$.
Let $b \in \bk$ be such that $(Y-b) \in \Galg(F,A)$; then $\deg F(X,b)=1$ and
$F(X,b) = G(Xv(b),b) = \sum_{i=0}^m G_i(b)(v(b)X)^i$;
since $m>1$ we have $G_m(b) v(b)^m =0$, so $v(b)=0$ and consequently $F(X,b) = G_0(b) \in \bk$, a contradiction.

Suppose that (ii) holds.
Let $d = \deg v(Y) \ge 1$ and
define an $\Nat$-grading $\bk[S,T] = \bigoplus_{i \in \Nat} R_i$ by stipulating that
$S$ is homogeneous of degree $d$ and $T$ is homogeneous of degree $1$.
Write $G(S,T) = \sum_i H_i(S,T)$ where $H_i(S,T) \in R_i$ for all $i$.
By (ii), we may choose $a \in \bk^*$ such that $\deg F(a,Y)=1$. Then
$$
F(a,Y) = G(av(Y),Y) = \sum_i H_i(av(Y),Y) \quad \text{where $\deg H_i(av(Y),Y) \le i$.}
$$
Let $m = \deg_S(G)$ and $N=\deg_T(G)$, and note that $m \ge2$ and $N \ge1$.
Since $G$ is a variable of $\bk[S,T]$,
\begin{equation}  \label {987fb7f452g2jhyf}
(m,0), (0,N) \in \supp_{(S,T)}(G) \subseteq \langle (0,0),(m,0),(0,N) \rangle .
\end{equation}
If $N \neq md$ then, by  \eqref{987fb7f452g2jhyf}, the highest $H_i(S,T)$ is either\footnote{We use Abhyankar's 
symbol ``$\abh$'' to represent an arbitrary element of $\bk^*$. Note that different occurrences of $\abh$ may represent
different elements of $\bk^*$.}
$\abh S^m$ (if $md>N$) or $\abh T^N$ (if $md<N$),
and in both cases we have $\deg F(a,Y) \ge md > 1$, a contradiction.
So $N = md$ and  the highest $H_i(S,T)$ is 
$H_{md} = ( \lambda T^d + \mu S )^m$  for some $\lambda,\mu \in \bk^*$.
Then for each $j \in \{1,2\}$ we have $\deg F(a_j,Y) = 1 < md$, so the right hand side of
$$
( \lambda Y^d + \mu a_j v(Y) )^m = F(a_j,Y) - \sum_{i<md} H_i(a_jv(Y),Y) 
$$
has degree less than $md$; then $\deg( \lambda Y^d + \mu a_j v(Y) ) < d$ for all $j\in\{1,2\}$,
but clearly there can be only one $a_j$ for which this holds.
This contradiction completes the proof.
\end{proof}

In the following, $\bk$ is an arbitrary field.

\begin{theorem}  \label {9854dnc2mrhvfc}
Let $F$ be a field generator of $A=\kk2$ satisfying $| \Gamma(F,A) | > 2$.
Then there exists $(X,Y)$ such that $A = \bk[X,Y]$ and $F = \alpha(Y)X+\beta(Y)$
for some $\alpha(Y), \beta(Y) \in \bk[Y]$.
\end{theorem}

\begin{proof}
We first prove the case where $\bk$ is algebraically closed.
Note the following consequence of
\ref{09239023r02n02b27c2c8h2}\eqref{09293bbgxqweiby7sznxadfefj}, which we will use several times:
\begin{equation} \label {76sd8k4dbvcxzrxf312}
\begin{minipage}[t]{.9\textwidth}
\it Suppose that $F$ is a rectangular element of $R = \kk2$ and that $D_1,D_2,D_3$ are distinct
elements of $\Gamma(F,R)$.
Then $D_i \cap D_j = \emptyset$ for some $i,j \in \{1,2,3\}$.
\end{minipage}
\end{equation}

We may assume that $F$ is not a variable of $A$, otherwise the conclusion is obvious.
Then $F$ is rectangular by \ref{982393298urcnj092}.
Choose $\gamma=(X,Y) \in \Rec(F,A)$, let $(m,n) = (\deg_X(F), \deg_Y(F))$ and recall that $\min(m,n) \ge 1$.
By \ref{fj29o9232g7qwjda}, $F$ is not a small field generator of $A$; 
so $F$ is not $\gamma$-small; interchanging $X$ and $Y$ if necessary, it follows that $F \in \bk[X v(Y), Y]$
for some $v(Y) \in \bk[Y] \setminus \bk$. Any such $v(Y)$ satisfies $\deg v(Y) \le n/m$,
so we may choose $v(Y) \in \bk[Y] \setminus \bk$ satisfying  $F \in \bk[X v(Y), Y]$ and:
\begin{equation} \label {73rtcg8trjfga}
\text{$\deg w(Y) \le \deg v(Y)$ for every $w(Y) \in \bk[Y]$ satisfying  $F \in \bk[X w(Y), Y]$.}
\end{equation}

Let $A_1 = \bk[X v(Y), Y]$ and let $\Phi : \Spec A \to \Spec A_1$ be the birational morphism
determined by the inclusion homomorphism $A_1 \hookrightarrow A$.
By \ref{56fewf8r8t34kd223lwgfuio}, there is an injective set map $\Gamma(F,A) \to \Gamma(F,A_1)$
given by $C \mapsto \Phi(C)$. 
Pick distinct elements $C_1, C_2, C_3 \in \Gamma(F,A)$;
then $\Phi(C_i)$ ($1 \le i \le 3$) are distinct elements of $\Gamma(F,A_1)$.
In particular, 
\begin{equation}  \label {90fj3226g3da1j}
| \Gamma(F,A_1) | \ge 3.
\end{equation}

We may assume that $F$ is not a variable of $A_1$,
otherwise the desired conclusion follows from \ref{qn03h7tg5rvbtc07yr}.
As $F$ is a field generator of $A_1$ which is not a variable,
$F$ is a rectangular element of $A_1$  by \ref{982393298urcnj092}.
Relabeling $C_1,C_2,C_3$ if necessary, we get
\begin{equation} \label {23fdhc9n5gsjz1}
\Phi(C_1) \cap \Phi(C_2) = \emptyset
\end{equation}
by \eqref{76sd8k4dbvcxzrxf312}.
In particular, $C_1 \cap C_2 = \emptyset$.
So, by \ref{09239023r02n02b27c2c8h2}\eqref{09293bbgxqweiby7sznxadfefj}, there exist
$Z \in \{X,Y\}$ and $\lambda_1,\lambda_2 \in \bk$ such that $C_1=V(Z-\lambda_1)$ and $C_2=V(Z-\lambda_2)$.
Let $b \in \bk$ be a root of $v(Y)$, then $\Phi$ contracts the line $V(Y-b) \subset \Spec A$ to a point.
By \eqref{23fdhc9n5gsjz1}, $V(Y-b)$ cannot meet both $C_1,C_2$, so
$$
\text{$C_1 = V(Y-\lambda_1)$ \quad and\quad $C_2 = V(Y-\lambda_2)$ for some distinct $\lambda_1,\lambda_2 \in \bk$.}
$$
For later use, let us also record that one of conditions \eqref{890dfghjk32467ewtvfccbmzc}, \eqref{0sdj76fg7k1hy} holds
(again by \ref{09239023r02n02b27c2c8h2}\eqref{09293bbgxqweiby7sznxadfefj}):
\begin{gather}
\label {890dfghjk32467ewtvfccbmzc} \text{$C_3 = V(Y-\lambda_3)$ for some $\lambda_3 \in \bk$;} \\
\label {0sdj76fg7k1hy} \text{$C_3 = V(X-\lambda_3)$ for some $\lambda_3 \in \bk$.}
\end{gather}

Pick $(X_1,Y_1) \in \Rec(F,A_1)$.
Then \ref{09239023r02n02b27c2c8h2}\eqref{09293bbgxqweiby7sznxadfefj} implies that,
for each $i = 1,2,3$, there exist $Z_i \in \{X_1,Y_1\}$ and $\mu_i \in \bk$ such that 
$\Phi(C_i) = V(Z_i-\mu_i)$.  It is easy to see that $\Phi(C_1) = V(Y-\lambda_1)$ in $\Spec A_1$, 
so $Z_1-\mu_1 = \abh(Y-\lambda_1)$.
So we can choose $(X_1,Y_1) \in \Rec(F,A_1)$ in such a way that $Y_1=Y$. Then
$\bk[Xv(Y),Y] = \bk[X_1,Y]$ and consequently $X_1 = \abh X v(Y) + P(Y)$ for some $P(Y) \in \bk[Y]$.
Multiplying $X_1$ and $P(Y)$ by a unit if necessary, we find that
$( X v(Y) + P(Y), Y ) \in \Rec(F,A_1)$ for some  $P(Y) \in \bk[Y]$.
We set:
$$
\gamma_1 = (X_1,Y_1) = ( X v(Y) + P(Y), Y ) \in \Rec(F,A_1) .
$$
By \eqref{90fj3226g3da1j} and \ref{fj29o9232g7qwjda}, $F$ is not a small field generator of $A_1$.
So $F$ is not $\gamma_1$-small and consequently one of 
\eqref{982qh33cg65r1exv}, \eqref{chdcx3s121dsw12zs4s5wcr} holds:
\begin{gather}
\label {982qh33cg65r1exv} \text{$F \in \bk[X_1 u(Y_1),\, Y_1]$ for some $u(Y_1) \in \bk[Y_1] \setminus \bk$;}\\
\label {chdcx3s121dsw12zs4s5wcr} \text{$F \in \bk[X_1,\, u(X_1) Y_1]$ for some $u(X_1) \in \bk[X_1] \setminus \bk$.}
\end{gather}
If \eqref{982qh33cg65r1exv}  holds then 
$F \in \bk[X_1 u(Y_1), Y_1] = \bk[(X v(Y) + P(Y)) u(Y), Y] = \bk[ Xv(Y)u(Y), Y]$,
which contradicts \eqref{73rtcg8trjfga}.
So \eqref{chdcx3s121dsw12zs4s5wcr}  must hold. Pick $c \in \bk$ such that $u(c)=0$; then
$F \in \bk[X_1, u(X_1) Y_1] \subseteq \bk[X_1, (X_1-c) Y_1]$. Let $A_2 = \bk[X_1, (X_1-c) Y_1]$
and consider the birational morphism $\Psi : \Spec A_1 \to \Spec A_2$ determined by $A_2 \hookrightarrow A_1$.
By \ref{56fewf8r8t34kd223lwgfuio}, $\Psi(\Phi(C_1))$, $\Psi(\Phi(C_2))$ and $\Psi(\Phi(C_3))$
are distinct elements of $\Gamma(F,A_2)$.
We claim:
\begin{equation} \label {98652f35sw2qfdh3267tr}
\text{$\Psi(\Phi(C_i)) \cap \Psi(\Phi(C_j)) \neq \emptyset$ for all choices of $i,j \in \{1,2,3\}$.}
\end{equation}
We prove this in each of the cases \eqref{890dfghjk32467ewtvfccbmzc} and \eqref{0sdj76fg7k1hy}.
Note that $\Psi$ contracts the line $L = V(X_1-c) \subset \Spec A_1$ to a point.
If \eqref{890dfghjk32467ewtvfccbmzc} holds then
$\Phi(C_i) = V(Y_1 - \lambda_i) \subset \Spec A_1$ $(i=1,2,3)$, so $L$ meets $\Phi(C_i)$ for $i=1,2,3$;
then $\Psi(\Phi(C_1)) \cap \Psi(\Phi(C_2)) \cap \Psi(\Phi(C_3)) \neq \emptyset$,
so \eqref{98652f35sw2qfdh3267tr} holds in this case.
If \eqref{0sdj76fg7k1hy} holds then 
$\Phi(C_i) = V(Y_1 - \lambda_i)$ for $i=1,2$; as $L$ meets these two lines, we get 
$\Psi(\Phi(C_1)) \cap \Psi(\Phi(C_2)) \neq \emptyset$.
Moreover, $C_i \cap C_3 \neq \emptyset$ for $i=1,2$, so 
$\Psi(\Phi(C_i)) \cap \Psi(\Phi(C_3)) \neq \emptyset$ for $i=1,2$,
and again \eqref{98652f35sw2qfdh3267tr} is true.
So \eqref{98652f35sw2qfdh3267tr} is true in all cases.

From \eqref{98652f35sw2qfdh3267tr} and \eqref{76sd8k4dbvcxzrxf312},
we deduce that $F$ is not a rectangular element of $A_2$.
As $F$ is a field generator of $A_2$, \ref{982393298urcnj092} implies that
\begin{equation*}
\text{$F$ is a variable of $A_2 = \bk[X_1, (X_1-c) Y_1]$.}
\end{equation*}
Then \ref{qn03h7tg5rvbtc07yr} gives $F = a_1 (X_1-c) Y_1 + a_2 X_1+a_3$
for some $a_1 \in \bk^*$, $a_2,a_3 \in \bk$. Then
$$
F = (X v(Y) + P(Y)) (a_1 Y + a_2) - a_1 c Y + a_3 ,
$$
from which the desired conclusion follows.
This proves the theorem in the case where $\bk$ is algebraically closed.

\medskip
To prove the general case, consider a field $\bk$ and
a field generator $F$ of $A=\kk2$ satisfying $| \Gamma(F,A) | > 2$.
We may assume that $F$ is not a variable of $A$, otherwise the conclusion is obvious.
Then $F$ is a rectangular element of $A$ by \ref{982393298urcnj092} and we may choose $(X,Y) \in \Rec^+(F,A)$.

Let $\ck$ be the algebraic closure of $\bk$ and $\bar A = \ck \otimes_\bk A = \ck^{[2]}$.
Then $F$ is a field generator of $\bar A$ and $(X,Y) \in \Rec^+(F,\bar A)$.
In particular, $F$ is not a variable of $\bar A$ (since it is a rectangular element of $\bar A$).
Note that $\bideg_{\bar A}(F) = (\deg_X(F), \deg_Y(F)) = \bideg_A(F)$.

We have $|\Gamma(F,\bar A)| \ge |\Gamma(F,A)| > 2$ by \ref{653b34hj24h75f7hk676g},
so the case ``$\bk$ algebraically closed'' of the 
Theorem implies that there exists $(X_1,Y_1)$ such that  $\bar A = \ck[X_1,Y_1]$
and $F = \alpha_1(Y_1)X_1 + \beta_1(Y_1)$ for some $\alpha_1(Y_1), \beta_1(Y_1) \in \ck[Y_1]$. 
Observe that $\deg \alpha_1(Y_1) > 0$, for otherwise $F$ would be a variable of $\bar A$.
Write $\beta_1(Y_1) = q(Y_1) \alpha_1(Y_1) + \rho(Y_1)$ with $q(Y_1), \rho(Y_1) \in \ck[Y_1]$ and
$\deg \rho(Y_1) < \deg \alpha_1(Y_1)$. Then
$F = \alpha_1(Y_1) ( X_1 + q(Y_1) ) + \rho(Y_1)
= \alpha_1(Y_2) X_2 + \rho(Y_2)$ where we define $X_2 = X_1 + q(Y_1)$ and $Y_2=Y_1$.
Since $\deg\rho(Y_2) < \deg \alpha_1(Y_2)$, we have $(X_2,Y_2) \in \Rec(F,\bar A)$.
This shows that $\bideg_{\bar A}(F) = (1,n)$ for some $n\ge1$.
Then $\bideg_{A}(F) = (1,n)$, so $F = \alpha(Y)X+\beta(Y)$ for some $\alpha(Y), \beta(Y) \in \bk[Y]$,
as desired.
\end{proof}

\begin{corollary}   \label {3q939xuh23829x9mh812o} 
If  $F$ is a bad field generator of $A = \kk2$ then the following hold.
\begin{enumerata}

\item \label {x23c3x5v4c5bv6nbm7nki0p} $| \Gamma(F,A) | \le 2$ 

\item \label {xvbvn5v5n7bm8m68bn43v53c}
$\Rec(F,A) \neq \emptyset$ and the pair $(m,n) = \bideg_A(F)$ satisfies $2 \le m \le n$.

\item \label {c-84397bcfdv563} There exists $(X,Y) \in \Rec(F,A)$ such that
$$
\Galg(F,A) \subseteq \{ (X), (Y) \} \quad \text{or} \quad \Galg(F,A) \subseteq \{ (X), (X-1) \}.
$$

\end{enumerata}
\end{corollary}

\begin{proof}
If $| \Gamma(F,A) | > 2$ then \ref{9854dnc2mrhvfc} implies that
there exists $(X,Y)$ such that $A = \bk[X,Y]$ and $F = \alpha(Y)X+\beta(Y)$
for some $\alpha(Y), \beta(Y) \in \bk[Y]$. This contradicts the hypothesis that $F$ is bad.
Indeed, if $\alpha(Y)=0$ then $F \in \bk[Y]$, so $F$ is not a bad field generator of $A$;
and if $\alpha(Y) \neq 0$ then  $\bk(F,Y)=\bk(X,Y)$, so again $F$ is good. This proves \eqref{x23c3x5v4c5bv6nbm7nki0p}.

Since $F$ is a field generator of $A$ which is not a variable,
$\Rec(F,A) \neq \emptyset$ by \ref{982393298urcnj092}.
The pair $(m,n) = \bideg_A(F)$ is defined and satisfies $1 \le m \le n$.
Since $F$ is bad, we must have $m>1$
(if $m=1$ then pick $(X,Y) \in \Rec^+(F,A)$; then $\deg_X(F)=m=1$, so $\bk(F,Y)=\bk(X,Y)$, so $F$ is good),
which proves \eqref{xvbvn5v5n7bm8m68bn43v53c}.

\eqref{c-84397bcfdv563} is an easy consequence of
$| \Gamma(F,A) | \le 2$ and \ref{09239023r02n02b27c2c8h2}\eqref{09293bbgxqweiby7sznxadfefj}.
\end{proof}

\begin{example}
Let $F = X(X-1)Y^2+Y \in A = \bk[X,Y]$.  Then $\bk(F,XY) = \bk(X,Y)$, so $F$ is a good field generator of $A$.
Moreover, 
\begin{equation}  \label {9f0289d1b12exenmsd}
\Galg(F,A) = \{ (X), (X-1) \}.
\end{equation}
At present, we don't know an example of a bad field generator satisfying \eqref{9f0289d1b12exenmsd}
(compare with \ref{3q939xuh23829x9mh812o}\eqref{c-84397bcfdv563}).
\end{example}

\begin{remark}
Let $\gamma = | \Galg(F,A) |$ where $F$ is a bad field generator of $A = \kk2$.
Then $\gamma \in \{0,1,2\}$ by \ref{3q939xuh23829x9mh812o}.
We shall see in \ref{92898fhuqnvalh893} that the three cases arise.
\end{remark}

\section{Very good and very bad field generators}
\label {SEC:VGVBFGs}

We begin by studying good and very good field generators.
We shall need the following fact, valid over an arbitrary field $\bk$:

\begin{observation}[{\cite[Rem.\ after 1.3]{Rus:fg}}] \label {82364r5gg658943k}
Let $F$ be a field generator of $A=\kk2$.
Then $F$ is a good field generator of $A$ if and only if \/\mbox{\rm ``$1$''} occurs in the list $\Delta(F,A)$.
\end{observation}

Also note the following consequence of \ref{982393298urcnj092}:

\begin{corollary} \label {8986ad8s89ddsf9d99sd9f}
Let $F$ be a field generator of $A=\kk2$ and $\Delta(F,A) = [d_1, \dots, d_t]$.
Then $F$ is a variable of $A$ if and only if $t=1$.
\end{corollary}

\begin{proof}
If $F$ is a variable then $\Delta(F,A) = [1]$, so $t=1$.
If $F$ is not a variable then, by \ref{982393298urcnj092}, there exists $(X,Y) \in \Rec(F,A)$.
The embedding $\aff^2 \hookrightarrow \proj^2$ determined by the pair $(X,Y)$ has the property that 
the closure of $V(F)$ in $\proj^2$ meets the line at infinity
in exactly two points.
Then $t\ge2$ by \ref{939ri09109e2xj9cru}.
\end{proof}

Let us agree that the $\gcd$ of the empty set is $+\infty$. Then:

\begin{proposition} \label {q23328r83yd74r9128}
Let $F$ be a field generator of $A=\kk2$ and $\Delta(F,A) = [d_1, \dots, d_t]$.
\begin{enumerata}

\item \label {a-734tc8rgj3rfgh}
If $\gcd\big( \{ d_1, \dots, d_t \} \setminus \{1\} \big)>1$
then $F$ is a very good field generator of $A$.
In particular,
if at most one $i \in \{1, \dots, t\}$ satisfies $d_i>1$ then $F$ is a very good field generator of $A$.

\item \label {b-734tc8rgj3rfgh}
If at least three $i \in \{1, \dots, t\}$ satisfy $d_i=1$ then $F$ is a very good field generator of $A$.

\item \label {c-734tc8rgj3rfgh} If $F$ is a good but not very good field generator of $A$
then 
$$
\Delta(F,A) = [1, \dots, 1, e_1, \dots, e_s]
$$
where \mbox{\rm ``$1$''} occurs either $1$ or $2$ times,
$s \ge 2$, $\min(e_1, \dots, e_s)>1$ and\\
$\gcd(e_1, \dots, e_s)=1$.

\end{enumerata}
\end{proposition}

\begin{proof}
Assertion \eqref{c-734tc8rgj3rfgh} follows from  \eqref{a-734tc8rgj3rfgh} and \eqref{b-734tc8rgj3rfgh}.
To prove  \eqref{a-734tc8rgj3rfgh} and \eqref{b-734tc8rgj3rfgh},
we consider $A'$ such that $F \in A' \preceq A$;
in each of cases \eqref{a-734tc8rgj3rfgh} and \eqref{b-734tc8rgj3rfgh}, it has to be shown that
$F$ is a good field generator of $A'$. Clearly, $F$ is a field generator of $A'$.

\eqref{a-734tc8rgj3rfgh} It follows from \ref{difja;skdfj;aksd} that $\Delta(F,A') = [d_1', \dots, d_s']$
is a sublist of $\Delta(F,A)$, so  $\gcd\big(\{d_1', \dots, d_s'\} \setminus \{1\}\big)>1$.
As $\gcd(d_1', \dots, d_s')=1$ by \ref{89329d675fd43q}, 
``$1$'' occurs in $\Delta(F,A')$, so \ref{82364r5gg658943k} implies that $F$ is a
good field generator of $A'$.

\eqref{b-734tc8rgj3rfgh} Consider the morphisms
$\Spec A \xrightarrow{\Phi} \Spec A' \xrightarrow{f} \Spec \bk[F]$
determined by the inclusions $\bk[F] \hookrightarrow A'  \hookrightarrow A$.
Let $C_1, \dots, C_h$ be the distinct elements of  $\Miss_{\text{\rm hor}}(\Phi,f)$ and,
for each $i \in \{ 1, \dots, h \}$,
let $\delta_i$ be the degree of $f|_{C_i} : C_i \to \Spec\bk[F]$.  Then \ref{difja;skdfj;aksd} gives
$\Delta(F,A) = \big[ \Delta(F,A'), \delta_1, \dots, \delta_h \big]$ and
\begin{equation} \label {x80293r81298r182hwddh}
\text{for all $i \in \{1,\dots, h\}$,\quad $\delta_i=1$ $\Leftrightarrow$ $C_i \in \Gamma(F,A')$.}
\end{equation}
Arguing by contradiction, assume that $F$ is a bad field generator of $A'$;
then (by \ref{3q939xuh23829x9mh812o}) $| \Gamma(F,A') | \le 2$, so (by \eqref{x80293r81298r182hwddh})
``$1$'' occurs at most twice in $[\delta_1, \dots, \delta_h]$;
as (by \ref{82364r5gg658943k}) ``$1$'' does not occur in $\Delta(F,A')$,
it follows that ``$1$'' occurs at most twice in  $\Delta(F,A)$, 
which contradicts the assumption of  \eqref{b-734tc8rgj3rfgh}.
So \eqref{b-734tc8rgj3rfgh} is proved.
\end{proof}

\begin{remark}  \label {pf9823f898012dj}
By~\ref{q23328r83yd74r9128}\eqref{a-734tc8rgj3rfgh},
the polynomials classified in \cite{MiySugie:GenRatPolys},
\cite{NeumannNorbury:simple} and \cite{Sasao_QuasiSimple2006} are special cases of very good field generators.
This gives many complicated examples of very good field generators.
\end{remark}

\begin{example} \label {87hf25f34w32a1a} 
Let $F$ be a bad field generator of $A = \Comp[X,Y]$ such that 
$\Galg(F,A) = \{ (X), (Y) \}$ and $\Delta(F,A) = [3,2]$
(such an $F$ exists by \ref{92898fhuqnvalh893}(c)).
Let $B = \Comp[X/Y,\, Y^2/X]$ and note that $A \preceq B$.
Consider the morphisms
$\Spec B \xrightarrow{\Phi} \Spec A \xrightarrow{f} \Spec \Comp[F]$
determined by the inclusions $\Comp[F] \hookrightarrow A  \hookrightarrow B$.
Then the missing curves of $\Phi$ are $C_1=V(X)$ and $C_2=V(Y)$ and these are $f$-horizontal,
so $\Miss_{\text{\rm hor}}(\Phi,f) = \{ C_1, C_2 \}$.
In the notation of \ref{difja;skdfj;aksd} we have $\delta_1=\delta_2=1$ (because $C_1,C_2 \in \Gamma(F,A)$),
so that result implies that $\Delta(F,B) = [3,2,1,1]$.
Note that {\it $F$ is not a very good field generator of $B$} (because it is bad in $A$).
This shows that, in \ref{q23328r83yd74r9128}\eqref{b-734tc8rgj3rfgh},
one cannot replace ``at least three'' by ``at least two'';
and in the second part of \ref{q23328r83yd74r9128}\eqref{a-734tc8rgj3rfgh}, one cannot replace
``at most one'' by ``at most two''.
\end{example}

\begin{example}  \label {ngru62r56mcutgk}
Let $\bk$ be any field and $F = XY \in A = \bk[X,Y] = \kk2$.
It is easy to see that $F$ is a good field generator of $A$ with $\Delta(F)=[1,1]$;
so by \ref{q23328r83yd74r9128}\eqref{a-734tc8rgj3rfgh} $F$ is a very good field generator of $A$.
Define $B = \bk \big[ X+Y,\, \frac Y{(X+Y)(X+Y+1)} \big]$ and note that $F \in A \preceq B$.
Consider the morphisms
$\Spec B \xrightarrow{\Phi} \Spec A \xrightarrow{f} \Spec \bk[F]$
determined by the inclusions $\bk[F] \hookrightarrow A  \hookrightarrow B$.
Then the missing curves of $\Phi$ are $C_1=V(X+Y)$ and $C_2=V(X+Y+1)$ and these are $f$-horizontal,
so $\Miss_{\text{\rm hor}}(\Phi,f) = \{ C_1, C_2 \}$.
In the notation of \ref{difja;skdfj;aksd} we have $\delta_1=\delta_2=2$
so that result implies that $\Delta(F,B) = [1,1,2,2]$.
By \ref{q23328r83yd74r9128}\eqref{a-734tc8rgj3rfgh}, $F$ is a very good field generator of $B$.
So there exist very good field generators with more than one dicritical of degree greater than $1$.
The polynomial $f$ given on page 298 of \cite{BartoCassou:RemPolysTwoVars} is another example, this one with
$\Delta(f, \Comp[X,Y]) = [1,1,2,4]$.
\end{example}

Observe that the problem of characterizing very good field generators is not settled by \ref{q23328r83yd74r9128};
in particular, we don't know whether the converse of \eqref{c-734tc8rgj3rfgh} is true.
However, very bad field generators can be characterized: 
result \ref{88x889adb823mdnfv}\eqref{n2n92bm2brxgftr}
gives such a characterization and, in fact,
makes it easy to decide whether a given bad field generator is very bad.
We shall derive that characterization from the following result:

\begin{proposition} \label {99z9zcvx61mn24hytf}
Consider $F \in A \preceq A'$, where $F$ is a field generator of $A$.
Let $\Phi : \Spec A' \to \Spec A$ be the morphism determined by $A \hookrightarrow A'$.
The following are equivalent:
\begin{enumerata}

\item \label {7634h2v5n32rvnv6rncddtgd}
$F$ is a bad field generator of $A'$;

\item \label {mnhfre13467hytz}
$F$ is a bad field generator of $A$ and no element of $\Gamma(F,A)$ is a missing curve of~$\Phi$.

\end{enumerata}
\end{proposition}

\begin{proof}
Consider the morphisms $\Spec A' \xrightarrow{\Phi} \Spec A \xrightarrow{f} \Spec \bk[F]$
determined by $\bk[F] \hookrightarrow A  \hookrightarrow A'$,
let $C_1, \dots, C_h$ be the distinct elements of $\Miss_{\text{\rm hor}}(\Phi,f)$ and,
for each $i \in \{ 1, \dots, h \}$,
let $\delta_i$ be the degree of the dominant morphism $f|_{C_i} : C_i \to \aff^1$.
We have $\Delta(F,A') = [\Delta(F,A), \delta_1, \dots, \delta_h ]$ by \ref{difja;skdfj;aksd},
and the same result also implies that
\begin{equation}  \label {q3o4ip23idj92jw}
\text{for all $i \in \{1, \dots, h\}$},\ \  \delta_i = 1 \iff C_i \in \Gamma(F,A).
\end{equation}
Now \eqref{7634h2v5n32rvnv6rncddtgd} is true
if and only if (by \ref{82364r5gg658943k})
``$1$'' does not occur in $\Delta(F,A') = [\Delta(F,A), \delta_1, \dots, \delta_h ]$,
if and only if (by \eqref{q3o4ip23idj92jw})
``$1$'' does not occur in $\Delta(F,A)$ and $C_i \notin \Gamma(F,A)$ for all $i$,
if and only if \eqref{mnhfre13467hytz} is true.
The last equivalence uses \ref{82364r5gg658943k} and the fact that, for any $C \in \Gamma(F,A)$,
$C \in \Miss(\Phi) \Leftrightarrow C \in \Miss_{\text{\rm hor}}(\Phi,f)$.
\end{proof}

\begin{proposition}  \label {88x889adb823mdnfv}
Let $F$ be a bad field generator of $A = \kk2$.
\begin{enumerata}

\item \label {n2n92bm2brxgftr}
$F$ is a very bad field generator of $A$ if and only if $\Galg(F,A) = \emptyset$.

\item Suppose that $\Galg(F,A) \neq \emptyset$.
Then there exists $(X,Y)$ such that $A = \bk[X,Y]$ and $(X) \in \Galg(F,A)$.
For any such pair $(X,Y)$, $F$ is a good field generator of $\bk[X, Y/X]$.

\end{enumerata}
\end{proposition}

\begin{proof}
We first prove (b).  Suppose that $\Galg(F,A) \neq \emptyset$ and pick $\pgoth \in \Galg(F,A)$.
As $F$ is not a variable of $A$, it is a rectangular element of $A$ (by \ref{982393298urcnj092}),
so \ref{0f29382c398y1hdq} implies that $\pgoth = (X)$ for some variable $X$ of $A$;
this shows that there exists $(X,Y)$ such that $A = \bk[X,Y]$ and $(X) \in \Galg(F,A)$.
Given such a pair $(X,Y)$, let $B = \bk[X, Y/X]$ and
consider the birational morphism $\Phi : \Spec B \to \Spec A$ determined by the inclusion $A \to B$.
Then $C = V(X) \subset \Spec A$ is the unique missing curve of $\Phi$ and
the fact that $(X) \in \Galg(F,A)$ implies that $C \in \Gamma(F,A)$.
By \ref{99z9zcvx61mn24hytf}, it follows that $F$ is a good field generator of $B$; so (b) is proved.

To prove (a), it's enough to show that 
\begin{equation} \label {09s0xq001287876tgd8splkkhi0r}
\text{$F$ is not a very bad field generator of $A$ $\implies$ $\Galg(F,A) \neq \emptyset$}
\end{equation}
is true, because we already know, by (b), that  the converse of \eqref{09s0xq001287876tgd8splkkhi0r} is true.
Suppose that $F$ is not very bad.
Then there exists a ring $B$ such that $A \preceq B$ and $F$ is a good field generator of $B$.
Consider the birational morphism $\Phi : \Spec B \to \Spec A$ determined by the inclusion $A \to B$.
By \ref{99z9zcvx61mn24hytf}, some element of $\Gamma(F,A)$ is a missing curve of $\Phi$;
so $\Gamma(F,A) \neq \emptyset$ (and hence $\Galg(F,A) \neq \emptyset$).
This proves \eqref{09s0xq001287876tgd8splkkhi0r}, from which (a) follows.
\end{proof}

\begin{lemma} \label {0sf09k35l53k7efjh834rth576}
Consider $F \in A \preceq A'$, where $F$ is a field generator of $A$,
and let $\Phi : \Spec A' \to \Spec A$ denote the morphism determined by $A \hookrightarrow A'$.
Consider the following conditions on the triple $(F,A,A')$:
\begin{enumerate}

\item[(i)] $F$ is a bad field generator of $A$ and,
for all $C \in \Gamma(F,A)$, $C$ is not a missing curve of $\Phi$ and $C \nsubseteq \image(\Phi)$;

\item[(ii)] $F$ is a bad field generator of $A$ and,
for all $C \in \Gamma(F,A)$, $C$ is not a missing curve of $\Phi$ and
no curve $D \subset \Spec A'$ satisfies $\Phi(D) = C$ and is such that $\Phi|_D : D \to C$ is an isomorphism;

\item[(iii)] $F$ is a very bad field generator of $A'$.

\end{enumerate}
Then $\mbox{\rm (i)} \Rightarrow \mbox{\rm (ii)} \Leftrightarrow \mbox{\rm (iii)}$.
\end{lemma}

\begin{proof}
Clearly, $\mbox{\rm (i)} \Rightarrow \mbox{\rm (ii)}$.
In view of \ref{99z9zcvx61mn24hytf} and \ref{56fewf8r8t34kd223lwgfuio}\eqref{324jhs7d5ghr2bv2npo24nc},
condition (ii) is equivalent to:
\begin{quote}
$F$ is a bad field generator of $A'$ 
and no element of $\Gamma(F,A)$ is in the image of the map $\gamma : \Gamma(F,A') \to \Gamma(F,A)$
defined in \ref{56fewf8r8t34kd223lwgfuio},
\end{quote}
which is equivalent to
\begin{quote}
$F$ is a bad field generator of $A'$ and $\Gamma(F,A') = \emptyset$.
\end{quote}
So, by \ref{88x889adb823mdnfv}, (ii)$\Leftrightarrow$(iii).
\end{proof}

\begin{proposition}  \label {0d023epoi012p94938yte6q}
For a field generator $F$ of $A = \kk2$, the following are equivalent.
\begin{enumerata}

\item $F$ is a bad field generator of $A$;

\item there exists $A' \succeq A$ such that $F$ is a very bad field generator of $A'$;

\item there exists $(X,Y)$ such that $A = \bk[X,Y]$ and $F$ is a very bad field generator
of $\bk[X,Y/X]$.

\end{enumerata}
\end{proposition}

\begin{proof}
Implications $\text{(c)} \Rightarrow \text{(b)} \Rightarrow \text{(a)}$ are obvious.
Assume that (a) holds and let us prove (c). 
By \ref{3q939xuh23829x9mh812o}, there exists a pair $(X,Y)$ satisfying $A = \bk[X,Y]$ and 
$$
\text{$\Galg(F,A) \subseteq \{ (X), (Y) \}$\qquad or\qquad (ii) $\Galg(F,A) \subseteq \{ (X), (X-1) \}$.}
$$
Define
$$
(X_1,Y_1) = \begin{cases}
(X-Y, Y-1) & \text{in case (i)}, \\
(X-Y^2+Y,Y) & \text{in case (ii)}.
\end{cases}
$$
Note that $A = \bk[X_1,Y_1]$ and define $A' = \bk[X_1,Y_1/X_1]$; then $F \in A \preceq A'$.
We claim that $F$ is a very bad field generator of $A'$.
To see this, consider the birational morphism
$\Phi : \Spec A' \to \Spec A$ determined by the inclusion $A \hookrightarrow A'$.
Consider the curve $D \subset \Spec A$ and the point $P \in \Spec A$ defined by
$D=V(X_1)$ and $\{P\} = V(X_1,Y_1)$;  then $D$ is the unique missing curve of $\Phi$ and
$\image\Phi = (\Spec(A) \setminus D) \cup \{P\}$.

In case (i), we have $V(X) \cap D \nsubseteq \{P\}$ and $V(Y) \cap D \nsubseteq \{P\}$;
in case (ii), $V(X) \cap D \nsubseteq \{P\}$ and $V(X-1) \cap D \nsubseteq \{P\}$.
So, in both cases, we have $C \cap D \nsubseteq \{P\}$ for all $C \in \Gamma(F,A)$;
consequently,
\begin{center}
$C \nsubseteq \image\Phi$ for all $C \in \Gamma(F,A)$.
\end{center}
Moreover, $D \notin \Gamma(F,A)$, so no element of $\Gamma(F,A)$ is a missing curve of $\Phi$.
So $(F,A,A')$ satisfies condition (i) of \ref{0sf09k35l53k7efjh834rth576};
by that result, $F$ is a very bad field generator of $A'$.
So condition~(c) holds.
\end{proof}

\medskip
To our knowledge, \cite{JanThesis}, \cite{Rus:fg2} and \cite{Cassou-BadFG} 
are the only publications giving examples of bad field generators.
Example~\ref{GoodBadUgly}, below, shows how the results of this paper can be used to easily produce 
new examples from old ones.
It also gives examples of very bad field generators, of bad field generators which are not very bad and
of good ones which are not very good.
(Very good field generators are easy to find: see \ref{pf9823f898012dj}.)

\begin{example} \label {GoodBadUgly}
Let $F \in A = \bk[X,Y] = \kk2$ (where $\bk$ is any field) be the following polynomial of degree~$21$:
\begin{equation*}
( Y^2(XY + 1)^4 + Y(2XY + 1)(XY + 1) + 1)
(Y(XY + 1)^5 +2XY(XY + 1)^2 +X)\,.
\end{equation*}
It is shown in \cite{Rus:fg2} (modulo a typo corrected in \cite{Cassou-BadFG}) that $F$ is a bad field generator of $A$
with $\Delta(F,A) = [2,3]$.
Note that $(X,Y) \in \Rec^+(F,A)$ and that $\bideg_A(F) = (9,12)$.
It is not difficult to deduce from \ref{09239023r02n02b27c2c8h2} that $\Galg(F,A) = \{ (Y) \}$.

If $\bk=\Comp$ then the fact that $\Delta(F,A)=[2,3]$ can also be deduced from the proof of \cite[2.3.10]{Cassou-BadFG},
essentially by noting that
\begin{align*}
F(-t^3+t^4u, 1/t^3)&=2u^3-3u^6+u^9+th_1(t,u) = \phi_1(u^3)+th_1(t,u), \\
F(1/t^6, -t^4+t^5/3+t^6u)&=36u^2+54u+20+th_2(t,u) = \phi_2(u)+th_2(t,u),
\end{align*}
where $\phi_1,\phi_2,h_1,h_2$ are polynomials, $\deg\phi_1=3$ and $\deg\phi_2=2$
(refer to the proof of \cite[2.3.10]{Cassou-BadFG}).

In the following three paragraphs we regard Russell's polynomial $F$ as an element of certain overrings of $A$.
By doing so,  we obtain new examples of field generators
(which could be called ``the good, the very bad, and the ugly'').

\medskip
\noindent (a) Since $\Galg(F,A) = \{ (Y) \}$, \ref{88x889adb823mdnfv} implies that
$F$ is not a very bad field generator of $A$.
By the same result, $F$ is a good (but not very good) field generator of $\bk[X/Y, Y]$.
In particular,
{\it there exist bad field generators that are not very bad,
and good field generators that are not very good.}

\medskip
\noindent (b) Let $A' = \bk[X, (Y-1)/X]$,
then $(F,A,A')$ satisfies condition (i) of \ref{0sf09k35l53k7efjh834rth576};
by that result, $F$ is a very bad field generator of $A'$.
So, {\it very bad field generators do exist}.
One can see that $\deg_{(X',Y')}(F)=33$, where $X'=X$ and $Y'=(Y-1)/X$. 
Using \ref{difja;skdfj;aksd}, one sees that $\Delta(F, A')=[2,3,3]$.

\medskip
\noindent (c) Choose distinct elements $\lambda_1, \dots, \lambda_N$ of $\bk$ ($N\ge1$)
and let $P(T) = \prod_{i=1}^N (T-\lambda_i)^{e_i} \in \bk[T]$
(where $T$ is an indeterminate and $e_i \ge 1$ for all $i$).  Pick $d\ge1$ and define 
$$
A'' = \bk[X+Y^d, Y/P(X+Y^d) ].
$$
Then $A = \bk[X,Y] = \bk[X+Y^d, Y] \preceq A''$.
Consider the morphisms
$\Spec A'' \xrightarrow{\Phi} \Spec A \xrightarrow{f} \Spec \bk[F]$
determined by the inclusions $\bk[F] \hookrightarrow A  \hookrightarrow A''$.
Then the missing curves of $\Phi$ are:
$$
C_i:\ \  X+Y^d=\lambda_i, \quad i = 1, \dots, N.
$$
Each $C_i$ is $f$-horizontal and the degree of $C_i \hookrightarrow \Spec A \xrightarrow{f} \Spec\bk[F]$ is
$\delta_i = \deg_t F( \lambda_i - t^d, t ) = 9d+12$.
So, by \ref{difja;skdfj;aksd},
\begin{equation} \label {09329r92x8y3r2d92a}
\text{$\Delta(F, A'') = [2,3,\, 9d+12, \dots, 9d+12 ]$ where ``$9d+12$'' occurs $N$ times.}
\end{equation}
By \ref{82364r5gg658943k}, $F$ is a bad field generator of $A''$.
If we write $A'' = \bk[u,v]$ with $u = X+Y^d$ and $v = Y/P(X+Y^d)$ then one can see that $\Galg(F,A'')= \{ (v) \}$,
so (by \ref{88x889adb823mdnfv}) $F$ is not a very bad field generator of $A''$.

Let $A''' = \bk[u, (v-1) / (u-\lambda_1) ]$
and let $\Spec A''' \xrightarrow{\Psi} \Spec A'' \xrightarrow{f''} \Spec\bk[F]$ be the morphisms determined by
$\bk[F] \hookrightarrow A'' \hookrightarrow A'''$.
Define $C \subset \Spec A''$ and $Q \in \Spec A''$ by $C = V(u-\lambda_1)$ and $\{Q\} = V(u-\lambda_1,v-1)$.
Then $C$ is the unique missing curve of $\Psi$ and $\image\Psi = ( \Spec(A'') \setminus C ) \cup \{Q\}$.
It follows that $(F,A'',A''')$ satisfies condition (i) of \ref{0sf09k35l53k7efjh834rth576};
by that result, $F$ is a very bad field generator of $A'''$.
As $f''(C)$ is a point, $\Miss_{\text{\rm hor}}(\Psi,f'') = \emptyset$;
so, by \ref{difja;skdfj;aksd}, $\Delta(F,A''') = \Delta(F,A'')$ (refer to \eqref{09329r92x8y3r2d92a}).
In particular, {\it there exist very bad field generators
with arbitrarily large numbers of dicriticals of arbitrarily large degrees.}
\end{example}

\begin{examples}  \label {92898fhuqnvalh893}
Let $\gamma = | \Galg(F,A) |$ where $F$ is a bad field generator of $A = \kk2$.
Then $\gamma \in \{0,1,2\}$ by \ref{3q939xuh23829x9mh812o}.
The three cases arise:
\begin{enumerata}

\item By \ref{88x889adb823mdnfv}, the condition $\gamma=0$ is equivalent to $F$ being very bad, and
very bad field generators do exist by \ref{GoodBadUgly}. So $\gamma=0$ arises.

\item The polynomial $F \in A = \bk[X,Y]$ given at the beginning of \ref{GoodBadUgly} (Russel's polynomial)
has $\Galg(F,A) = \{ (Y) \}$, so  $\gamma=1$.

\item It is stated in \cite[2.3.11]{Cassou-BadFG} and proved in \cite[2.3.10]{Cassou-BadFG} that the polynomial
\begin{multline*}
F = X(X^5Y^3 +1)^3 +Y(X^2Y+1)^8 - X^{16}Y^9 +4XY+6X^2Y+19X^3Y^2 \\
 +8X^4Y^2 + 36X^5Y^3 + 34X^7Y^4 + 16X^9Y^5
\end{multline*}
is a bad field generator of $A = \Comp[X,Y]$.
It follows from the proof of \cite[2.3.10]{Cassou-BadFG} that $\Delta(F,A) = [3,2]$
(indeed, in the setting of that proof, this follows by noting that
\begin{align*}
F(1/t^3,-t^5+t^6u) &=27u^3+72u^2+66u+20+tk_1(t,u) = \phi_1(u) + tk_1(t,u), \\
F(it^4+t^5u, 1/t^8) &=256u^8+161u^4-1+tk_2(t,u) = \phi_2(u^4) + tk_2(t,u),
\end{align*}
where $\phi_1,\phi_2,k_1,k_2$ are polynomials, $\deg\phi_1=3$ and $\deg\phi_2=2$).
It is obvious that $(X,Y) \in \Rec(F,A)$ and that $\bideg_A(F) = (\deg_YF, \deg_XF) = (9,16)$.
It follows from \ref{09239023r02n02b27c2c8h2} that $\Galg(F,A) = \{ (X), (Y) \}$,
so $\gamma=2$.

\end{enumerata}
\end{examples}


\end{document}